\documentclass[12pt]{amsart}
\usepackage{amsmath, amsthm, amsfonts, amssymb, color}
\usepackage{mathrsfs}
\usepackage{amsfonts, amsmath}
\usepackage{bbm}

\theoremstyle{plain}
\newtheorem{theorem}{Theorem}[section]
\newtheorem{prop}[theorem]{Proposition}
\newtheorem{lemma}[theorem]{Lemma}

\theoremstyle{definition}
\newtheorem{definition}[theorem]{Definition}
\newtheorem{remark}[theorem]{Remark}

\DeclareMathOperator{\supp}{supp}

\allowdisplaybreaks
\def\RR{\mathbb R}

\def\11{1\!\!1}
\def\supp{\text{supp}}

\newcommand{\bmo}{{\rm BMO}}

\newcommand{\XX}{X}

\medskip

\title[Product Hardy spaces]{Product Hardy spaces associated to operators
with heat kernel bounds on spaces of homogeneous type}

\author{Peng Chen,\ Xuan Thinh Duong,\ Ji Li,\ Lesley A.~Ward\ and\ Lixin Yan}

\address{
Peng Chen, Department of Mathematics, Sun Yat-sen University, Guangzhou, 510275, China, and
School of Information Technology and Mathematical Sciences, University of South Australia,
Mawson Lakes SA 5095, Australia}
\email{chenpeng3@mail.sysu.edu.cn}
\address{
Xuan Thinh Duong, Department of Mathematics, Macquarie University, NSW 2109, Australia}
\email{
xuan.duong@mq.edu.au}
\address{
Ji Li, Department of Mathematics, Macquarie University, NSW 2109, Australia}
\email{
ji.li@mq.edu.au}
\address{
Lesley A.~Ward,
School of Information Technology and Mathematical
Sciences, University of South Australia, Mawson Lakes SA 5095,
Australia}
\email{
lesley.ward@unisa.edu.au}
\address{
Lixin Yan, Department of Mathematics, Sun Yat-sen University, Guangzhou, 510275, China}
\email{
mcsylx@mail.sysu.edu.cn
}

\subjclass[2010]{42B20, 42B25, 46B70, 47G30.}
\keywords{Singular integrals, Hardy spaces, product space,
atomic decomposition, Calder\'on--Zygmund decomposition.}

\date{}

\begin{document}

\maketitle

\medskip

\begin{abstract}
    The aim of this article is to develop the theory of product
    Hardy spaces associated with operators which possess the
    weak assumption of Davies--Gaffney heat kernel estimates,
    in the setting of spaces of homogeneous type. We also
    establish a Calder\'on--Zygmund decomposition on product
    spaces, which is of independent interest, and use it to
    study the interpolation of these product Hardy spaces. We
    then show that under the assumption of generalized Gaussian
    estimates, the product Hardy spaces coincide with the
    Lebesgue spaces, for an appropriate range of~$p$.
\end{abstract}

\maketitle

\tableofcontents

%-----------------------------------------------------------------------
\section{Introduction }
\setcounter{equation}{0}

Modern harmonic analysis was introduced in the '50s, with the
Calder\'on--Zygmund theory at the heart of it. This theory
established criteria for singular integral operators to be
bounded on different scales of function spaces, especially the
Lebesgue spaces $L^p$, $1 < p < \infty$. To achieve this goal,
an integrated part of the Calder\'on--Zygmund theory includes
the theory of interpolation and the theory of function spaces,
in particular end-point spaces such as the Hardy and BMO
spaces. The development of the theory of Hardy spaces in
$\mathbb{R}^n$ was initiated by E.M.~Stein and G.~Weiss
\cite{SW}, and was originally tied to the theory of harmonic
functions. Real-variable methods were introduced into this
subject by C.~Fefferman and E.M.~Stein in~\cite{FS}; the
evolution of their ideas led eventually to characterizations of
Hardy spaces via the atomic or molecular
decomposition. See for instance \cite{C}, %\cite{L},
\cite{St} and~\cite{TW}. The advent of the atomic and molecular
characterizations enabled the extension of the Hardy spaces on
Euclidean spaces to the more general setting of spaces of
homogeneous type~\cite{CW}.

%This theory of Hardy spaces has been very successful and fruitful in the last few decades.

While the Calder\'on--Zygmund theory with one parameter was
well established in the four decades of the '50s to '80s,
multiparameter Fourier analysis was introduced later in the
'70s and studied extensively in the '80s by a number of well
known mathematicians, including R.~Fefferman, S.-Y. A. Chang,
R. Gundy, E.M. Stein, and J.L. Journ\'e (see for instance \cite
{CF1}, \cite{CF2}, \cite{CF3}, \cite{F1}, \cite{F2}, \cite{F3},
\cite{F4}, \cite{FSt}, \cite{GS}, \cite{Jo}).

It is now understood that there are important situations in
which the standard theory of Hardy spaces is not applicable and
there is a need to consider Hardy spaces that are adapted to
certain linear operators, similarly to the way that the
standard Hardy spaces are adapted to the Laplacian. In this new
development, the real-variable techniques of \cite{CW},
\cite{FS} and \cite{CMS} are still of fundamental importance.

Recently, a theory of Hardy spaces associated to operators was
introduced and developed by many authors. The following are
some previous closely related results in the one-parameter
setting.

%\smallskip

(i) In \cite{AuDM}, P. Auscher, X.T. Duong and A. M$^{\rm c}$Intosh
introduced  the Hardy space $H_L^1(\mathbb{R}^n)$ associated to an
operator $L$, and obtained a molecular decomposition, assuming that
$L$ has a bounded holomorphic functional calculus on
$L^2(\mathbb{R}^n)$ and the kernel of the heat semigroup $e^{-tL}$
has a pointwise Poisson upper bound.

(ii) Under the same assumptions on $L$ as in (i), X.T. Duong
and L.X. Yan introduced the space $\bmo_L(\mathbb{R}^n)$
adapted to $L$ and established the duality of
$H_L^1(\mathbb{R}^n)$ and $\bmo_{L^*}(\mathbb{R}^n)$ in
\cite{DY1}, \cite{DY2}, where $L^*$ denotes the adjoint
operator of $L$ on $L^2(\mathbb{R}^n)$. L.X. Yan \cite{Yan}
also studied the Hardy space $H_L^p(\mathbb{R}^n)$ and duality
associated to an operator $L$ under the same assumptions as
(ii) for  all $0<p<1$.

(iii)  P. Auscher,  A. M$^{\rm c}$Intosh and E.
Russ~\cite{AMR}, and S. Hofmann and S. Mayboroda~\cite{HM},
treated Hardy spaces $H^p_L, p\geq 1$, (and in the latter
paper, also BMO spaces) adapted, respectively, to the Hodge
Laplacian on a Riemannian manifold with a doubling measure, or
to a second order divergence form elliptic operator on
$\mathbb{R}^n$ with complex coefficients, in which settings
pointwise heat kernel bounds may fail.

%By making  use of a
%notion of ``$L$-cancellation'' of molecules, they  studied  the
%Hardy space $H^1_L$ including a molecular decomposition, a area
%function characterization, its dual space and others properties.

(iv) S. Hofmann, G. Lu, D. Mitrea, M. Mitrea and L.X.
Yan~\cite{HLMMY} developed the theory of $H^1_L(X)$ and ${\rm
BMO}_L(X)$ spaces adapted to a non-negative, self-adjoint
operator $L$ whose heat semigroup satisfies the weak
Davies--Gaffney bounds, in the setting of a space of
homogeneous type $X$.

%For the Hardy space $H^1_L$, they
%also obtained an atomic decomposition.

(v) P.C. Kunstmann and M.~Uhl~\cite{KU1, U} studied the Hardy
spaces $H^p_L(X)$, $1<p<\infty$, associated to operators~$L$
satisfying the same conditions as in~(iv) as well as the
generalized Gaussian estimates for $p_0\in [1,2)$, and proved
that $H^p_L(X)$ coincides with $L^p(X)$ for $p_0 < p < p'_0$
where $p'_0$ is the conjugate index of $p_0$.

(vi) X.T. Duong and J.~Li~\cite{DL} considered the Hardy spaces
$H_L^p(X)$, $0 < p \leq 1$, associated to non-self-adjoint
operators $L$ that generate an analytic semigroup on $L^2(X)$
satisfying Davies--Gaffney estimates and having a bounded
holomorphic functional calculus on $L^2(X)$.

In contrast to the above listed established one-parameter
theory, the multiparameter theory is much more complicated and
is less advanced. In particular, there has not been much
progress in the last decade in the direction of the
paper~\cite{DY2} on singular integral operators with non-smooth
kernels and the related product function spaces.

In \cite{DSTY}, D.G. Deng, L. Song, C.Q. Tan and L.X. Yan
introduced the product Hardy space $H^1_L(\mathbb{R}\times
\mathbb{R})$ associated with an operator~$L$, assuming that $L$
has a bounded holomorphic functional calculus on
$L^2(\mathbb{R})$ and the kernel of the heat semigroup
$e^{-tL}$ has a pointwise Poisson upper bound.

Recently, X.T. Duong, J. Li and L.X. Yan \cite{DLY} defined the
product Hardy space $H^1_{L_1,L_2}(\mathbb{R}^{n_1}\times
\mathbb{R}^{n_2})$ associated with non-negative self-adjoint
operators $L_1$ and $L_2$ satisfying Gaussian heat kernel
bounds, and then obtained the atomic decomposition, as well as
the $H^1_{L_1,L_2}(\mathbb{R}^{n_1}\times
\mathbb{R}^{n_2})\rightarrow L^1(\mathbb{R}^{n_1}\times
\mathbb{R}^{n_2})$  boundedness of product singular integrals
with non-smooth kernels.

In the study of Hardy spaces $H^p$ associated to operators,
$1\leq p<\infty$, the assumptions on these operators determine
the relevant properties of the corresponding Hardy spaces. One
would start with the definition of Hardy spaces associated to
operators under ``weak" conditions on the operators so that the
definition is applicable to a large class of operators.
However, to obtain useful properties such as the coincidence
between the Hardy spaces $H^p$ and the Lebesgue spaces $L^p$,
one would expect stronger conditions on the operators are
needed. A natural question is to find a weak condition that is
still sufficient for the Hardy spaces and Lebesgue spaces to
coincide. We do so in part~($\gamma$) below.

This article is devoted to the study of Hardy spaces associated
to operators, in the setting of product spaces of homogeneous
type. Assume that $L_1$ and $L_2$ are two non-negative
self-adjoint operators acting on $L^2(X_1)$ and $L^2(X_2)$,
respectively, where $X_1$ and $X_2$ are spaces of homogeneous
type, satisfying Davies--Gaffney
estimates~\eqref{DaviesGaffney} (see
Section~\ref{sec:GGE_DG_FPS}{\bf(c)}). We note that the
Davies--Gaffney estimates are a rather weak assumption, as they
are known to be satisfied by quite a large class of operators
(see Section~\ref{sec:GGE_DG_FPS} below).

Our main results are the following. In this paper we work in
the biparameter setting. However our results, methods and
techniques extend to the full $k$-parameter setting.

($\alpha$) We define the product Hardy space
$H^1_{L_1,L_2}(X_1\times X_2)$ associated with $L_1$ and $L_2$,
in terms of the area function, and then obtain the
corresponding atomic decomposition (Theorem~\ref{theorem of
Hardy space atomic decom}). This is the generalisation of the
results in \cite{DLY} from the product of Euclidean spaces
under the stronger assumption of Gaussian
estimates~\eqref{Gaussian} (see
Section~\ref{sec:GGE_DG_FPS}{\bf(a)}) to the product of spaces
of homogeneous type with the weaker assumption of
Davies--Gaffney estimates~\eqref{DaviesGaffney}. This is also
the extension of~\cite{HLMMY} from the one-parameter setting to
the multiparameter setting. This part is the content of
Section~\ref{sec:atomic_decomposition}.

($\beta$) We define the product Hardy space
$H^p_{L_1,L_2}(X_1\times X_2)$ for $1 < p < \infty$ associated
with $L_1$ and $L_2$, and prove the interpolation result that
if an operator $T$ is bounded on $L^2(X_1\times X_2)$ and is
also bounded from $H_{L_1,L_2}^1(X_1\times X_2)$ to
$L^1(X_1\times X_2)$, then it is bounded from
$H_{L_1,L_2}^p(X_1\times X_2)$ to $L^p(X_1\times X_2)$ for all
$p$ with $1\leq p\leq 2$ (Theorem~\ref{theorem interpolation
Hp}). The proof of this interpolation result relies on the
Calderon--Zygmund decomposition in the product setting,
obtained in Theorem~\ref{theorem C-Z decomposition for Hp}
below, which generalizes the classical result of Chang and
Fefferman~\cite{CF2} on $H^1(\mathbb{R}\times \mathbb{R})$.
This is done in
Section~\ref{sec:CZ_decomposition_interpolation}.

($\gamma$) Next we assume that $L_1$ and $L_2$ satisfy
generalized Gaussian estimates (see
Section~\ref{sec:GGE_DG_FPS}{\bf (b)}) for some $p_0 \in
[1,2)$. This assumption implies that $L_1$ and $L_2$ are
injective operators (see Theorem \ref{theorem injective}) and
satisfy the Davies--Gaffney estimates. We prove that our
product Hardy spaces $H_{L_1,L_2}^p(X_1\times X_2)$ coincide
with $L^p(X_1\times X_2)$ for $p_0 < p < p'_0$, where $p'_0$ is
the conjugate index of $p_0$ (Theorem \ref{theorem-Hp-Lp}).
This is the extension to the multiparameter setting of the
one-parameter result in~\cite{U}, and is carried out in
Section~\ref{sec:HpandLp}.

Along this line of research,  in~\cite{CDLWY} we study the
boundedness of multivariable spectral multipliers on product
Hardy spaces on spaces of homogeneous type.

\smallskip

In the following section we introduce our assumptions on the
underlying spaces $X_1$ and $X_2$ and the operators $L_1$ and
$L_2$, give some examples of such operators, and then state our
main results. Throughout this article, the symbols ``$c$" and
``$C$" denote constants that are independent of the essential
variables.

%-----------------------------------------------------------------------
\section{Assumptions, and statements of main results}
\label{sec:assumptions_main_results}
\setcounter{equation}{0}

This section contains background material on spaces of
homogeneous type, dyadic cubes, heat kernel bounds, and finite
propagation speed of solutions to the wave equation, as well as
the statements of our main results.

%, and
%spectral multipliers of non-negative self-adjoint operators.

\newpage

%-----------------------------------------------------------------------
\subsection{Spaces of homogeneous type}

\begin{definition}\label{def:space_of_homog_type}
Consider a set $X$ equipped with a quasi-metric~$d$ and a
measure~$\mu$.
\begin{enumerate}
    \item[(a)] A \emph{quasi-metric}~$d$ on a set~$\XX$ is
        a function
        $d:\XX\times\XX\longrightarrow[0,\infty)$
        satisfying (i) $d(x,y) = d(y,x) \geq 0$ for all
        $x$, $y\in\XX$; (ii) $d(x,y) = 0$ if and only if $x
        = y$; and (iii) the \emph{quasi-triangle
        inequality}: there is a constant $A_0\in
        [1,\infty)$ such that for all $x$, $y$, $z\in\XX$, %\vspace{-.2cm}
        \begin{eqnarray*}\label{eqn:quasitriangleineq}
            d(x,y)
            \leq A_0 [d(x,z) + d(z,y)].
        \end{eqnarray*}

        We define the quasi-metric ball by $B(x,r) :=
        \{y\in X: d(x,y) < r\}$ for $x\in X$ and $r > 0$.
        Note that the quasi-metric, in contrast to a
        metric, may not be H\"older regular and
        quasi-metric balls may not be open.

        \item[(b)] We say that a nonzero measure $\mu$
            satisfies the \emph{doubling condition} if
            there is a constant $C$ such that for all
            $x\in\XX$ and $r > 0$,
            \begin{eqnarray}\label{doubling condition}
               %0<
               \mu(B(x,2r))
               \leq C \mu(B(x,r))
               < \infty.
            \end{eqnarray}

        \item[(c)] We point out that the doubling condition
            (\ref{doubling condition}) implies that there
            exist positive constants $n$ and $C$ such that
            for all $x\in X$, $\lambda\geq 1$ and $r > 0$,
            \begin{eqnarray}\label{upper dimension}
                \mu(B(x, \lambda r))
                \leq C\lambda^{n} \mu(B(x,r)).
            \end{eqnarray}
            Fix such a constant $n$; we refer to this $n$
            as \emph{the upper dimension of $\mu$}.

        \item[(d)] We say that $(\XX,d,\mu)$ is a {\it
            space of homogeneous type} in the sense of
            Coifman and Weiss if $d$ is a quasi-metric
            on~$\XX$ and $\mu$ is a nonzero measure on~$X$
            satisfying the doubling condition.
        \end{enumerate}
\end{definition}
Throughout the whole paper, we assume that $\mu(X) = +\infty$.
\medskip

It is shown in~\cite{CW} that every space of homogeneous type
$X$  is \emph{geometrically doubling}, meaning that there is
some fixed number~$T$ such that each ball~$B$ in~$X$ can be
covered by at most $T$ balls of half the radius of~$B$.

\medskip

We recall the following construction given by M.~Christ in
\cite{Chr}, which provides an analogue on spaces of homogeneous
type of the grid of Euclidean dyadic cubes. The following
formulation is taken from~\cite{Chr}.

%A slightly weaker version of the dyadic decomposition was also
%established independently by Sawyer-Wheeden in \cite{SW}.

\begin{lemma}[\cite{Chr}] \label{lemma-dyadic-cubes}
    Let $(X,d,\mu)$ be a space of homogeneous type. Then
    there exist a collection $\{I_\alpha^k \subset X: k\in
    \mathbb{Z}, \alpha \in \mathcal{I}_k\}$ of open subsets
    of~$X$, where $\mathcal{I}_k$ is some index set, and
    constants $C_3 < \infty$, $C_4 > 0$, such that
    \begin{enumerate}
    \item[(i)] $\mu(X\setminus \bigcup_\alpha I_\alpha^k) =
        0$ for each fixed $k$, and $I_\alpha^k\cap
        I_\beta^k = \emptyset$ if $\alpha\neq\beta$;

    \item[(ii)] for all $\alpha,\beta,k,l$ with $l\geq k$,
        either $I_\beta^l\subset I_\alpha^k$ or
        $I_\beta^l\cap I_\alpha^k = \emptyset$;

    \item[(iii)] for each $(k,\alpha)$ and each $l<k$ there
        is a unique $\beta$ such that $ I_\alpha^k \subset
        I_\beta^l $;

    \item[(iv)] ${\rm diam}(I_\alpha^k)\leq C_3 2^{-k} $;
        and

    \item[(v)] each $I_\alpha^k$ contains some ball
        $B(z_\alpha^k, C_4 2^{-k} )$, where $z_\alpha^k \in
        X$.
    \end{enumerate}
\end{lemma}

The point $z^k_\alpha$ is called the $\emph{centre}$ of the
set~$I^k_\alpha$. Informally, we can think of $I_\alpha^k$ as a
dyadic cube with diameter roughly $2^{-k}$, centered
at~$z_\alpha^k$. We write $\ell(I^k_\alpha) := C_3 2^{-k}$.

Given a constant $\lambda > 0$, we define $\lambda I_\alpha^k$
to be the ball
\[
    \lambda I_\alpha^k
    := B(z_\alpha^k,\lambda C_3 2^{-k});
\]
if $\lambda > 1$ then $I^k_\alpha\subset \lambda I_\alpha^k$.
We refer to the ball $\lambda I_\alpha^k$ as the \emph{cube
with the same center as $I_\alpha^k$ and diameter $\lambda {\rm
diam}(I_\alpha^k)$}, or as the \emph{$\lambda$-fold dilate} of
the cube~$I_\alpha^k$. Since $\mu$ is doubling, we have
$\mu(\lambda I^k_\alpha) \leq C \mu(B(z^k_\alpha,C_4 2^{-k}))
\leq C \mu(I^k_\alpha)$.

%-----------------------------------------------------------------------
\subsection{Generalized Gaussian estimates, Davies--Gaffney estimates,
and finite propagation speed}
\label{sec:GGE_DG_FPS}

Suppose that $L$ is a non-negative self-adjoint operator on
$L^2(X)$, and that the semigroup $\{e^{-tL}\}_{t>0}$ generated
by $L$ on $L^2(X)$ has the kernel $p_{t}(x,y)$.

\medskip

{\bf (a) Gaussian estimates:}  The kernel $p_{t}(x,y)$ has
\emph{Gaussian upper bounds}~\eqref{Gaussian} if there are
positive constants $C$ and $c$ such that for all $x$, $y\in X$
and all $t > 0$,
\begin{equation}\label{Gaussian}
    |p_{t}(x,y)|
    \leq \frac{C}{V(x,t^{1/2})}
        \exp\Big(-\frac{d(x,y)^2}{c\,t}\Big).
    \tag{GE}
\end{equation}

%Suppose that $L$ is a non-negative self-adjoint operator on
%$L^2(X)$, with semigroup $\{e^{-tL}\}_{t>0}$ generated by $L$
%on $L^2(X)$.

\smallskip
{\bf (b) Generalized Gaussian estimates:}\ We say that
$\{e^{-tL}\}_{t>0}$ satisfies the {\it generalized Gaussian
estimates}~\eqref{generalGE}, for a given $p\in[1,2]$, if there
are positive constants $C$ and~$c$ such that for all $x$, $y\in
X$ and all $t > 0$,
\begin{equation}\label{generalGE}
    \|P_{B(x,t^{1/2})}e^{-tL}P_{B(y,t^{1/2})}\|_{L^p(X)\to L^{p'}(X)}
    \leq C V(x,t^{1/2})^{-(1/p-1/p')}\exp\Big(-\frac{d(x,y)^2}{c\,t}\Big),
    \tag{${\rm GGE}_{p}$}
\end{equation}
where $1/p + 1/p' = 1$.

\smallskip
{\bf (c) Davies--Gaffney estimates:} We say that
$\{e^{-tL}\}_{t>0}$ satisfies the \emph{Davies--Gaffney
condition}~\eqref{DaviesGaffney} if there are positive
constants $C$ and~$c$ such that for all open subsets
$U_1,\,U_2\subset X$ and all $t > 0$,
\begin{equation}\label{DaviesGaffney}
    |\langle e^{-tL}f_1, f_2\rangle|
    \leq C\exp\Big(-\frac{{\rm dist}(U_1,U_2)^2}{c\,t}\Big)
        \|f_1\|_{L^2(X)}\|f_2\|_{L^2(X)}
    \tag{DG}
\end{equation}
for every $f_i\in L^2(X)$ with $\mbox{supp}\,f_i\subset U_i$,
$i=1,2$. Here ${\rm dist}(U_1,U_2) := \inf_{x\in U_1, y\in U_2}
d(x,y)$.

\smallskip
{\bf (d) Finite propagation speed:}
%\begin{definition}\label{def-of-FSP}
    We say that $L$ satisfies the \emph{finite propagation
    speed  property}~\eqref{FPS} for solutions of the
    corresponding wave equation if for all open sets
    $U_i\subset X$ and all $f_i\in L^2(U_i)$, $i = 1$, 2, we
    have
    \begin{equation}\label{FPS}
        \langle \cos(t\sqrt{L})f_1,f_2\rangle = 0 \tag{FS}
    \end{equation}
    for all $t\in(0,d(U_1,U_2))$. %$0 < c_0t < d(U_1,U_2)$
    %for $U_i\subset X$, $f_i\in L^2(U_i)$, $i=1,2$.
%\end{definition}

As the following lemma notes, it is known that the
Davies--Gaffney estimates and the finite propagation speed
property are equivalent. For the proof, see for
example~\cite[Theorem~3.4]{CS2008}.

\begin{lemma}\label{FSDG}
    Let $L$ be a non-negative self-adjoint operator acting on
    $L^2(X)$. Then the finite propagation speed
    property~\eqref{FPS} and the Davies--Gaffney
    estimates~\eqref{DaviesGaffney} are equivalent.
\end{lemma}

\begin{remark}
    Note that when $p=2$, it is shown in \cite[Lemma 3.1]{CS2008}
    that the generalized Gaussian estimates are the same as the
    Davies--Gaffney estimates~\eqref{DaviesGaffney}. Also, when
    $p=1$, the generalized Gaussian estimates~\eqref{generalGE}
    are equivalent to the Gaussian estimates~\eqref{Gaussian}
    (see~\cite[Proposition~2.9]{BK1}). By H\"older's inequality, we see
    that if an operator satisfies the generalized Gaussian estimates
    for some $p$ with $1 < p < 2$, then it also satisfies the
    generalized Gaussian estimates $({\rm GGE}_{q})$ for all $q$
    with $p < q \le 2$. In particular,
    \begin{eqnarray*}
        \eqref{Gaussian}
        \iff  \eqref{generalGE} \textup{ with } p=1
        \implies  \eqref{generalGE} \textup{ with } p\in(1,2]
        \implies \eqref{DaviesGaffney}
        \iff \eqref{FPS}.
    \end{eqnarray*}

    We also note that if the generalized Gaussian estimates
    \eqref{generalGE} hold for some $p\in[1,2)$, then the
    operator $L$ is injective on $L^2(X)$ (see Theorem~\ref{theorem
    injective}).
\end{remark}

Suppose $L$ is a non-negative self-adjoint operator acting on
$L^2(X)$, and satisfying the Davies--Gaffney
estimates~\eqref{DaviesGaffney}. Then the following result
holds.

\begin{lemma}[Lemma~3.5, \cite{HLMMY}]\label{lemma finite speed}
    Let $\varphi\in C^{\infty}_0(\mathbb R)$ be an even
    function with $\supp\,\varphi \subset (-1,
    1)$. Let $\Phi$ denote the Fourier
    transform of~$\varphi$. Then for every
    $\kappa=0,1,2,\dots$, and for every $t>0$, the kernel
    $K_{(t^2L)^{\kappa}\Phi(t\sqrt{L})}(x,y)$ of the operator
    $(t^2L)^{\kappa}\Phi(t\sqrt{L})$, which is defined via
    spectral theory, satisfies
    \begin{eqnarray}\label{e3.11}
        {\rm supp}\ \! K_{(t^2L)^{\kappa}\Phi(t\sqrt{L})}(x,y)
        \subseteq \Big\{(x,y)\in X\times X: d(x,y)\leq t\Big\}.
    \end{eqnarray}
\end{lemma}

\medskip

\noindent {\bf Examples.} \ We now describe some operators
where property~\eqref{FPS} and the estimates~\eqref{generalGE}
hold for some $p$ with $1\leq p < 2$.

Let $V\in L^1_{\rm loc}({\mathbb R}^n)$ be a non-negative
function. The Schr\"odinger operator with potential $V$ is
defined by $L=-\Delta+V$ on ${\mathbb R}^n$, where $n\geq3$.
From the well-known Trotter--Kato product formula, it follows
that the semigroup $e^{-tL}$ has a kernel $p_t(x,y)$ satisfying
\begin{equation}\label{e8.4}
    0
    \leq  p_t(x,y)
    \leq (4\pi t)^{-{n\over 2}}\exp\Big(-\frac{|x-y|^2}{4t}\Big)\ \  \
    {\rm for\ all\ }\  t>0, \ \ x,y\in{\Bbb R}^n.
\end{equation}
\noindent See~\cite[page~195]{Ou}. It follows that
property~\eqref{FPS} and the estimates~\eqref{generalGE} hold
with $p=1.$

Next we consider inverse square potentials, that is $V(x) = c/|
x |^2$.  Fix $n \geq 3$ and assume that $c >  -{(n-2)^2/4}$.
Define $L := -\Delta + V$ to be the standard quadratic form on
$L^2(\RR^n, dx)$. The classical Hardy inequality
\begin{equation}\label{hardy1}
    - \Delta\geq  {(n-2)^2\over 4}|x|^{-2},
\end{equation}
shows that for all $c > -{(n-2)^2/4}$,  the self-adjoint
operator $L$ is non-negative. Set $p_c^{\ast} := n/\sigma$, and
$\sigma := \max\{ (n-2)/2-\sqrt{(n-2)^2/4+c}, 0\}$. If $c \ge
0$ then the semigroup $\exp(-tL)$ is pointwise bounded by the
Gaussian semigroup and hence acts on all $L^p$ spaces with $1
\le p \le \infty$.  If $ c < 0$, then $\exp(-tL)$ acts as a
uniformly bounded semigroup on  $L^p(\RR^n)$ for $ p \in
((p_c^{\ast})', p_c^{\ast})$ and the range
 $((p_c^{\ast})', p_c^{\ast})$ is optimal (see for example~\cite{LSV}).
It is known (see for instance~\cite{CS2008}) that $L$ satisfies
property~\eqref{FPS} and the estimates~\eqref{generalGE} for
all $p \in ((p_c^{\ast})', 2n/(n+2)]$. If $c \ge 0$, then $p =
(p_c^{\ast})' = 1$ is included.

It is also known (see \cite{KU}) that the
estimates~\eqref{generalGE} hold for some $p$ with $1\leq p <
2$ (and hence the property~\eqref{FPS} also holds) when $L$ is
the second order Maxwell operator with measurable coefficient
matrices, or the Stokes operator with Hodge boundary conditions
on bounded Lipschitz domains in ${\mathbb R}^3$, or the
time-dependent Lam\'e system equipped with homogeneous
Dirichlet boundary conditions.

%-----------------------------------------------------------------------
\subsection{Main results: Product Hardy spaces associated with
operators} We begin this section by defining the Hardy space
$H^2({X_1\times X_2})$. Next we introduce the area function
$Sf$, and use it to define the Hardy space
$H^1_{{L_1,L_2}}({X_1\times X_2})$ associated to $L_1$ and
$L_2$ (Definition~\ref{def of Hardy space via S function}). We
define $(H^1_{L_1, L_2}, 2, N)$-atoms $a(x_1,x_2)$
(Definition~\ref{def H1 atom}) and use them to define the
atomic Hardy space $H^1_{L_1, L_2, at,N}({X_1\times X_2})$
(Definition~\ref{def-of-atomic-product-Hardy-space}). We show
that these two definitions of this Hardy space coincide
(Theorem~\ref{theorem of Hardy space atomic decom}). We also
define the Hardy space $H^p_{L_1,L_2}(X_1\times X_2)$
associated to $L_1$ and $L_2$, via a modified area function
(Definition~\ref{def2.2}). We present the Calder\'on--Zygmund
decomposition of the Hardy spaces $H^p_{L_1,L_2}(X_1\times
X_2)$ (Theorem~\ref{theorem C-Z decomposition for Hp}). We use
this decomposition to establish two interpolation results and
to show that $H^p_{L_1,L_2}(X_1\times X_2)$ coincides with
$L^p(X_1\times X_2)$ for an appropriate range of~$p$
(Theorems~\ref{theorem interpolation Hp}
and~\ref{theorem-Hp-Lp}).

We work with the product of spaces of homogeneous
type~$(X_1,d_1,\mu_1)\times (X_2,d_2,\mu_2)$.  Here, for $i =
1$, 2, $(X_i,d_i,\mu_i)$ is a space of homogeneous type with
upper dimension~$n_i$, as in
Definition~\ref{def:space_of_homog_type}, and $\mu_i(X_i)=\infty$.

Following \cite{AMR}, one can define the \emph{$L^2({X_1\times
X_2})$-adapted Hardy space}
\begin{equation}\label{eq2.H2}
    H^2({X_1\times X_2})
    := \overline{R(L_1\otimes L_2)},
\end{equation}
that is, the closure of the range of $L_1\otimes L_2$ in
$L^2({X_1\times X_2})$.  Then $L^2({X_1\times X_2})$ is the
orthogonal sum of $H^2({X_1\times X_2})$ and the null space
$N(L_1\otimes L_2)=\{ f\in L^2({X_1\times X_2}): (L_1\otimes
L_2)f=0 \}$.

We shall work with the domain $(X_{1}\times \mathbb{R}_+)
\times (X_{2}\times \mathbb{R}_+)$ and its distinguished
boundary ${X_1\times X_2}$. For $x = (x_1,x_2)\in {X_1\times
X_2}$, denote by $\Gamma(x)$ the product cone $\Gamma(x) :=
\Gamma_1(x_1)\times\Gamma_2(x_2)$, where $\Gamma_i(x_i) :=
\{(y_i,t_i)\in X_{i}\times \mathbb{R}_+: d_i(x_i,y_i) < t_i\}$
for $i = 1$, 2.

Our first definition of the product Hardy space
$H^1_{L_1,L_2}(X_1\times X_2)$ associated to operators $L_1$
and~$L_2$ is via an appropriate area function. For $i = 1$, 2,
suppose that $L_i$ is a non-negative self-adjoint operator on
$X_i$ such that the corresponding heat semigroup $e^{-tL_i}$
satisfies the Davies--Gaffney estimates~\eqref{DaviesGaffney}.
Given a function $f$ on $L^2({X_1\times X_2})$, the \emph{area
function}~$Sf$ associated with the operators $L_1$ and $L_2$ is
defined by
\begin{eqnarray}\label{esf}
    \hskip.7cm Sf(x)
    := \bigg(\iint_{\Gamma(x) }\big|\big(
        t_1^2L_1e^{-t_1^2L_1} \otimes t_2^2L_2e^{-t_2^2L_2}\big)f(y)\big|^2\
        {d\mu_1(y_1) dt_1 d\mu_2(y_2) dt_2 \over t_1V(x_1,t_1) t_2V(x_2,t_2)}\bigg)^{1/2}.
   \end{eqnarray}
Since $L_1$ and $L_2$ are non-negative self-adjoint operators,
it is known from $H_\infty$ functional calculus~\cite{M}  that there exist
constants $C_1$ and $C_2$ with $0 < C_1 \leq C_2 < \infty$ such
that
$$  \|Sf\|_2\leq C_2\|f\|_2  $$
for all $f\in L^2(X_1\times X_2)$, and (by duality)
$$  C_1\|f\|_2\leq \|Sf\|_2 $$
for all $f\in H^2(X_1\times X_2)$.

%It is known that there exist constants $C_1$, $C_2$ with $0 <
%C_1 \leq C_2 < \infty$ such that for all $f\in L^2({X_1\times
%X_2})$,
%\begin{eqnarray}\label{S-function bd on L2}
%    C_1\|f\|_2
%    \leq \|S(f)\|_2\leq C_2\|f\|_2.
%\end{eqnarray}

\begin{definition} \label{def of Hardy space via S function}
    For $i = 1$, 2, let $L_i$ be a non-negative self-adjoint
    operator on $L^2(X_i)$ such that the corresponding heat
     semigroup $e^{-tL_i}$ satisfies the Davies--Gaffney
     estimates~\eqref{DaviesGaffney}. The
    \emph{Hardy space $H^1_{{L_1,L_2}}({X_1\times X_2})$ associated
    to $L_1$ and $L_2$} is defined as the completion of the set
    \[
        \{f\in H^2({X_1\times X_2}) :
        \|Sf\|_{L^1({X_1\times X_2})} < \infty\}
    \]
    with respect to the norm
    \[
        \|f\|_{H^{1}_{{L_1,L_2}}({X_1\times X_2}) }
        := \|Sf \|_{L^1({X_1\times X_2})}.
    \]
\end{definition}

%\begin{definition} \label{def of Hardy space via S function}
%    For $i = 1$, 2, let each $L_i$ be a non-negative
%    self-adjoint operator on $L^2(X_i)$ such that the
%    corresponding heat semigroup $e^{-tL_i}$ satisfy the
%    Davies--Gaffney estimates~\eqref{DaviesGaffney}. The
%    \emph{Hardy space $H^1_{{L_1,L_2}}({X_1\times X_2})$
%    associated to $L_1$ and $L_2$} is defined as the completion
%    of the set
%    \[
%        \{f\in L^2({X_1\times X_2}) :
%        \|S(f)\|_{L^1({X_1\times X_2})} < \infty\}
%    \]
%    with respect to the norm
%    \[
%        \|f\|_{H^{1}_{{L_1,L_2}}({X_1\times X_2}) }
%        := \|S(f) \|_{L^1({X_1\times X_2})}.
%    \]
%    Here $S(f)$ is the Littlewood--Paley area function; see
%    equation~\eqref{esf}.
%\end{definition}

We now introduce the notion of $(H^1_{L_1, L_2}, 2, N)$-atoms
associated to operators.

\begin{definition}\label{def H1 atom}
    Let $N$ be a positive integer. A function $a(x_1, x_2)\in
    L^2({X_1\times X_2})$ is called an $(H^1_{L_1, L_2}, 2,
    N)$-\emph{atom} if it satisfies the following conditions:
    \begin{enumerate}
        \item there is an open set $\Omega$ in ${X_1\times
            X_2}$ with finite measure such that
            $\supp\,a\subset \Omega$; and

        \item $a$ can be further decomposed as
        \[
            a = \sum\limits_{R\in m(\Omega)} a_R,
        \]
        where $m(\Omega)$ is the set of all maximal dyadic
        rectangles contained in $\Omega$, and for each
        $R\in m(\Omega)$ there exists a function $b_R$ such
        that for all $\sigma_1$, $\sigma_2 \in \{0, 1,
        \ldots, N\}$, $b_R$ belongs to the range of
        $L_1^{\sigma_1}\otimes L_2^{\sigma_2}$ in
        $L^2({X_1\times X_2})$ and
            \begin{enumerate}
                \item[(i)] $a_R = \big(L_1^{N} \otimes
                    L_2^{N}\big) b_R$;

                \item[(ii)] $\supp\,\big(L_1^{\sigma_1}
                    \otimes
                    L_2^{\sigma_2}\big)b_R\subset
                    \overline{C}R$;

                \item[(iii)] $||a||_{L^2( {X_1\times X_2}
                    )}\leq \mu(\Omega)^{-1/2}$ and
                    \[
                        \sum_{R = I\times J \in m(\Omega)}
                            \ell(I)^{-4N} \ell(J)^{-4N}
                            \Big\|\big(\ell(I)^2 \, L_1\big)^{\sigma_1}\otimes
                            \big(\ell(J)^2 \, L_2\big)^{\sigma_2}
                            b_R\Big\|_{L^2({X_1\times X_2})}^2
                        \leq \mu(\Omega)^{-1}.
                    \]
            \end{enumerate}
        \end{enumerate}
\end{definition}

Here $R = I\times J$, $\overline{C}$ is a fixed constant, and
$\overline{C}R$ denotes the product $\overline{C}I\times
\overline{C}J$ of the balls which are the $\overline{C}$-fold
dilates of $I$ and~$J$ respectively, as defined in Section~3.

We can now define an atomic $H^1_{L_1,L_2,at,N}$ space, which
we shall show is equivalent to the space $H^1_{L_1,L_2}$
defined above via area functions.

\begin{definition}\label{def-of-atomic-product-Hardy-space}
    Let $N$ be a positive integer with $N > \max\{n_1,
    n_2\}/4$, where $n_i$ is the upper doubling
    dimension of~$X_i$, $i = 1$,~2. %The \emph{Hardy space}
    %$H^1_{L_1, L_2, at,N}({X_1\times X_2})$ is defined as follows.
    We say that $f = \sum\lambda_ja_j$ is an \emph{atomic
    $(H^1_{L_1, L_2}, 2, N)$-representation of $f$} if
    $\{\lambda_j\}_{j=0}^\infty\in \ell^1$, each $a_j$ is an
    $(H^1_{L_1, L_2}, 2, N)$-atom, and the sum converges in
    $L^2({X_1\times X_2})$. Set
    \[
        \mathbb{H}^1_{L_1, L_2, at, N}({X_1\times X_2})
        := \big\{f: f~\text{has an atomic
            $(H^1_{L_1, L_2}, 2, N)$-representation}\big\},
    \]
    with the norm given by
    \begin{eqnarray}\label{Hp norm}
    &&\|f\|_{\mathbb{H}^1_{L_1, L_2, at, N}({X_1\times X_2})}\\
     &&\
    \ :=\inf\Big\{ \sum_{j=0}^\infty |\lambda_j|:
        f = \sum_j\lambda_ja_j~\text{is an atomic
        $(H^1_{L_1, L_2}, 2, N)$-representation} \Big\}.\nonumber
    \end{eqnarray}
    %
    %\begin{align}\label{Hp norm}
    %    &\lefteqn \|f\|_{\mathbb{H}^1_{L_1, L_2, at, N}({X_1\times X_2})} \hspace{1cm}\\
    %    &:=\inf\Big\{ \sum_{j=0}^\infty |\lambda_j|:
    %    f = \sum_j\lambda_ja_j~\text{is an atomic
    %    $(H^1_{L_1, L_2}, 2, N)$-representation} \Big\}.\nonumber
    %\end{align}
    The \emph{Hardy space} $H^1_{L_1, L_2, at,N}({X_1\times X_2})$
    is then defined as the completion of $\mathbb{H}^1_{L_1,
    L_2,at,N}({X_1\times X_2})$ with respect to this norm.
\end{definition}

%\begin{definition}\label{def-of-atomic-product-Hardy-space}
%    Let $N > \max\{n_1, n_2\}/4$, where $n_i$ is the upper
%    dimension of~$X_i$, for $i = 1$,~2 (see
%    Definition~\ref{def:space_of_homog_type}(c)). %The \emph{Hardy space}
%    %$H^1_{L_1, L_2, at,N}({X_1\times X_2})$ is defined as follows.
%    We say that $f = \sum\lambda_ja_j$ is an \emph{atomic
%    $(H^1_{L_1, L_2}, 2, N)$-representation of $f$} if
%    $\{\lambda_j\}_{j=0}^\infty\in \ell^1$, each $a_j$ is an
%    $(H^1_{L_1, L_2}, 2, N)$-atom, and the sum converges in
%    $L^2({X_1\times X_2})$. Set
%    \[
%        \mathbb{H}^1_{L_1, L_2, at, N}({X_1\times X_2})
%        := \big\{f: f~\text{has an atomic
%            $(H^1_{L_1, L_2}, 2, N)$-representation}\big\},
%    \]
%    with the norm given by
%    \begin{eqnarray}\label{Hp norm}
%    &&\|f\|_{\mathbb{H}^1_{L_1, L_2, at, N}({X_1\times X_2})}\\
%     &&\
%    \ :=\inf\Big\{ \sum_{j=0}^\infty |\lambda_j|:
%        f = \sum_j\lambda_ja_j~\text{is an atomic
%        $(H^1_{L_1, L_2}, 2, N)$-representation} \Big\}.\nonumber
%    \end{eqnarray}
%    %
%    %\begin{align}\label{Hp norm}
%    %    &\lefteqn \|f\|_{\mathbb{H}^1_{L_1, L_2, at, N}({X_1\times X_2})} \hspace{1cm}\\
%    %    &:=\inf\Big\{ \sum_{j=0}^\infty |\lambda_j|:
%    %    f = \sum_j\lambda_ja_j~\text{is an atomic
%    %    $(H^1_{L_1, L_2}, 2, N)$-representation} \Big\}.\nonumber
%    %\end{align}
%    The \emph{Hardy space} $H^1_{L_1, L_2, at,N}({X_1\times X_2})$
%    is then defined as the completion of $\mathbb{H}^1_{L_1,
%    L_2,at,N}({X_1\times X_2})$ with respect to this norm.
%\end{definition}

Our first result is that the ``area function" and ``atomic"
$H^1$ spaces coincide, with equivalent norms, if the parameter
$N > \max\{n_1, n_2\}/4$.

\begin{theorem}\label{theorem of Hardy space atomic decom}
    Let $(X_i,d_i,\mu_i)$ be spaces of homogeneous type with
    upper dimension~$n_i$, for $i = 1$, $2$. Suppose $N
    > \max\{n_1, n_2\}/4$. Then
    \[
        H^1_{{L_1,L_2}}({X_1\times X_2})
        = H^1_{L_1,L_2,at,N}({X_1\times X_2}).
    \]
    Moreover,
    \[
        \|f\|_{H^1_{{L_1,L_2}}({X_1\times X_2})}
        \sim \|f\|_{H^1_{L_1,L_2,at,N}({X_1\times X_2})},
    \]
    where the implicit constants depend only on $N$, $n_1$
    and~$n_2$.
\end{theorem}

It follows that
Definition~\ref{def-of-atomic-product-Hardy-space} always
yields the same Hardy space $H^1_{L_1, L_2, at,N}({X_1\times
X_2})$, independent of the particular choice of $N>\max\{n_1,
n_2\}/4$.

The proof of Theorem~\ref{theorem of Hardy space atomic decom}
will be given in Section 3.

\smallskip We turn from the case of $p = 1$ to the Hardy spaces
$H^p_{L_1,L_2}(X_1\times X_2)$ associated to $L_1$ and~$L_2$,
for $1 < p < \infty$.

\begin{definition}{\label{def2.2}}
    Let $L_1$ and $L_2$ be two non-negative, self-adjoint
    operators acting on $L^2({X_1})$ and $L^2({X_1})$
    respectively, satisfying the Davies--Gaffney condition {\rm
    (DG)}.

    (i) For each $p$ with $1 < p \leq 2$, the \emph{Hardy space
    $H^p_{L_1,L_2}(X_1\times X_2)$ associated to $L_1$ and $L_2$} is
    the completion of the space $\Big\{ f\in H^2({X_1\times X_2}):
    \ Sf\in L^p({X_1\times X_2})\Big\}$ in the norm
    $$
     \|f\|_{H_{L_1,L_2}^p(X_1,X_2)} = \|Sf\|_{L^p(X_1,X_2)}.
    $$

    (ii) For each $p$ with $2 < p < \infty$, the \emph{Hardy space
    $H^p_{L_1,L_2}(X_1,X_2)$ associated to $L_1$ and $L_2$} is the
    completion of the space $D_{K_0, p}$ in the norm
    $$
    \|f\|_{H_{L_1,L_2}^p(X_1,X_2)}
    := \|S_{K_0}f\|_{L^p(X_1,X_2)}, \quad\text{with}\quad
    K_0
    := \max\Big\{\left[\,{n_1\over 4}\,\right],\left[\,{n_2\over 4}\,\right]\Big\} + 1,
    $$
    where
    \begin{align}\label{S function for H p}
        S_{K}f(x)
        := \Big(\int_{\Gamma(x)}
            |(t_1^2L_1)^{K}e^{-t_1^2L_1}
            \otimes(t_2^2L_2)^{K}e^{-t_2^2L_2} f(y)|^2
            {d\mu_1(y_1)\over V(x_1,t_1)}
            {dt_1\over t_1}
            {d\mu_2(y_2)\over V(x_2,t_2)}
            {dt_2\over t_2}\Big)^{1/2},
    \end{align}
    and
    \begin{eqnarray*}
        D_{K, p}
        := \Big\{ f\in H^2({X_1\times X_2}):
            \ S_{K}f\in L^p({X_1\times X_2})\Big\}.
    \end{eqnarray*}
\end{definition}

%For $x\in {X_1\times X_2} $,
%where $f\in L^2({X_1\times X_2})$. For each $K\geq 1$ and $1<
%p<\infty$, we now define

Next we develop the Calder\'on--Zygmund decomposition of the
Hardy spaces $H^p_{L_1,L_2}({X_1\times X_2})$, which is a
generalization of the result of Chang and Fefferman~\cite{CF2}.

\begin{theorem}\label{theorem C-Z decomposition for Hp}
    Fix $p$ with $1 < p < 2$. Take $\alpha > 0$ and
    $f\in H_{L_1,L_2}^p(X_1\times X_2)$. Then we may write $f =
    g + b$, where $g\in H_{L_1,L_2}^{2}( X_1\times X_2 )$ and
    $b\in H_{L_1,L_2}^{1}( X_1\times X_2 )$, such that
    \[
        \|g\|^{2}_{H_{L_1,L_2}^{2}( X_1\times X_2 )}
        \le C\alpha^{2-p}\|f\|^p_{H_{L_1,L_2}^p( X_1\times X_2 )}
    \]
    and
    \[
        \|b\|_{H_{L_1,L_2}^{1}( X_1\times X_2 )}
        \le C\alpha^{1-p}\|f\|^p_{H_{L_1,L_2}^p( X_1\times X_2 )}.
    \]
    Here $C$ is an absolute constant.
\end{theorem}

As a consequence of the above Calder\'on--Zygmund
decomposition, we obtain the following interpolation result.

\begin{theorem}\label{theorem interpolation Hp}
    Suppose that $L_1$ and $L_2$ are non-negative self-adjoint
    operators such that the corresponding heat semigroups
    $e^{-tL_1}$ and $e^{-tL_2}$ satisfy the Davies--Gaffney
    estimates~\eqref{DaviesGaffney}. Let $T$ be a sublinear
    operator which is bounded on $L^2(X_1\times X_2)$ and
    bounded from $H_{L_1,L_2}^{1}( X_1\times X_2 )$ to $L^{1}(
    X_1\times X_2 )$. Then $T$ is bounded from $H_{L_1,L_2}^p(
    X_1\times X_2 )$ to $L^p( X_1\times X_2 )$ for all $1<p<2$.
\end{theorem}

The proofs of Theorems~\ref{theorem C-Z decomposition for Hp}
and~\ref{theorem interpolation Hp} will be given in
Section~\ref{sec:CZ_decomposition_interpolation}.

Next, we establish the relationship between the Hardy spaces
$H^p_{L_1,L_2}({X_1\times X_2})$ and the Lebesgue spaces
$L^p({X_1\times X_2})$ for a certain range of~$p$.

First note that under the assumption of Gaussian upper
bounds~\eqref{Gaussian}, following the approaches used
in~\cite{HLMMY} in the one-parameter setting, we can obtain
that $H^p_{L_1,L_2}({X_1\times X_2})=L^p({X_1\times X_2})$ for
all $1<p<\infty$. Second, if one assumes only the
Davies--Gaffney estimates on the heat semigroups of $L_1$ and
$L_2$, then for $1 < p < \infty$ and $p\not= 2$,
$H^p_{L_1,L_2}({X_1\times X_2})$ may or may not coincide with
the space $L^p({X_1\times X_2})$. An example where the
classical Hardy space can be different from the Hardy space
associated to an operator $L$ is when $L$ is the elliptic
divergence form operator with complex, bounded measurable
coefficients on $\mathbb{ R}^n$; see \cite{HM}. However, it can
be verified by spectral analysis that $H^2_{L_1,L_2}({X_1\times
X_2}) = H^2({X_1\times X_2})$. Here the $L^2({X_1\times
X_2})$-adapted Hardy space $H^2({X_1\times X_2})$ is as defined
in~\eqref{eq2.H2} above.
%as
%\begin{equation}\label{eq2.H2}
%H^2({X_1\times X_2}) := \overline{R(L_1\otimes L_2)},
%\end{equation}
%that is,
%the closure of the range of $L_1\otimes L_2$ in $L^2({X_1\times
%X_2})$, which implies that $L^2({X_1\times X_2})$ is the
%orthogonal sum of $H^2({X_1\times X_2})$ and the null space
%$N(L_1\otimes L_2)$.

\begin{theorem}\label{theorem-Hp-Lp}
    Suppose that $L_1$ and $L_2$ are non-negative self-adjoint
    operators on $L^2(X_1)$ and $L^2(X_2)$, respectively.
    %such that the corresponding heat semigroups satisfy
    %the Davies--Gaffney estimates \eqref{DaviesGaffney}.
    %Additionally, s
    Suppose that there exists some $p_0\in[1,2)$ such
    that $L_1$ and $L_2$ satisfy the generalized Gaussian
    estimates ${\rm (GGE_{p_0})}$.  Let $p'_0$
    satisfy $1/p_0 + 1/p'_0 = 1$.

    \begin{itemize}
    \item[(i)] We have $ H^p_{L_1,L_2}({X_1\times X_2}) =
        L^p({X_1\times X_2}) $ for all $p$ such that
        $p_0<p<p'_0$, with equivalent norms
        $\|\cdot\|_{H^p_{L_1,L_2}}$ and $\|\cdot\|_{L^p}$.

    \item[(ii)] Let $T$ be a sublinear operator which is
        bounded on $L^2(X_1\times X_2)$ and bounded from
        $H_{L_1,L_2}^{1}( X_1\times X_2 )$ to $L^{1}(
        X_1\times X_2 )$. Then $T$ is bounded on $L^p(
        X_1\times X_2 )$ for all $p$ such that
        $p_0<p<p'_0$.
    \end{itemize}
\end{theorem}

The proof of Theorem~\ref{theorem-Hp-Lp} will be given in
Section 5.
\bigskip

%-----------------------------------------------------------------------
\section{Characterization of the Hardy space $H^1_{L_1,L_2}({X_1\times X_2})$
in terms of atoms}
\setcounter{equation}{0} \label{sec:atomic_decomposition}

The goal of this section is to provide the proof of
Theorem~\ref{theorem of Hardy space atomic decom}.

Our strategy is as follows: by density, it is enough to show
that when $N > \max\{n_1,n_2\}/4$, we have
\[
    \mathbb{H}^1_{L_1,L_2,at,N}({X_1\times X_2})
    = H^1_{L_1,L_2 }({X_1\times X_2})\cap L^2({X_1\times X_2})
\]
with equivalent norms. The proof of this fact proceeds in two
steps.

\medskip

\noindent {\bf Step 1.} \
$\mathbb{H}^1_{L_1,L_2,at,N}({X_1\times X_2})\subseteq
H^1_{L_1,L_2}({X_1\times X_2})\cap L^2({X_1\times X_2})$, for
$N > \max\{n_1,n_2\}/4$. This step relies on the fact that the
area function $S$ is bounded on $L^2({X_1\times X_2})$ and that
$\|Sa\|_{L^1({X_1\times X_2})}$ is uniformly bounded for every
atom~$a$.

\medskip

\noindent {\bf Step 2.} \ $ H^1_{L_1,L_2}({X_1\times X_2})\cap
L^2({X_1\times X_2})\subseteq
\mathbb{H}^1_{L_1,L_2,at,N}({X_1\times X_2})$, for all
$N\in{\mathbb N}$. In the proof of this step we use the tent
space approach to construct the atoms in the Hardy spaces
associated to operators in the product setting.

\medskip

We take these in order.

\begin{proof}[Proof of Step 1]  The conclusion of Step 1 is an immediate
consequence of the following pair of Lemmata.
%\end{proof}

\begin{lemma}\label{lemma-of-T-bd-on product Hp}
    Fix $N\in{\mathbb N}$. Assume that $T$ is a linear
    operator, or a non-negative sublinear operator, satisfying the
    weak-type~\textup{(2,2)} bound
    \begin{eqnarray*}
    \big|\{x\in  {X_1\times X_2} : |Tf(x)|>\eta\}\big|
    \leq C\eta^{-2}\|f\|_{L^2( {X_1\times X_2} )}^2,\
    \ \text{for all}~\eta > 0,
    \end{eqnarray*}
    and that for every $(H^1_{L_1, L_2}, 2, N)$-atom $a$, we have
    \begin{eqnarray}%\label{e4.11} same label as in the following lemma???
    \|Ta\|_{L^1( {X_1\times X_2} )}\leq C
    \end{eqnarray}
    with constant $C$ independent of $a$. Then $T$ is bounded from
    $\mathbb{H}^1_{L_1,L_2,at,N}( {X_1\times X_2} )$ to $L^1(
    {X_1\times X_2} )$, and
    \[
        \|Tf\|_{L^1( {X_1\times X_2} )}
        \leq C\|f\|_{\mathbb{H}^1_{L_1,L_2,at,N}(X)}.
    \]
    Therefore, by density, $T$ extends to a bounded
    operator from
    $H^1_{L_1,L_2,at,N}( {X_1\times X_2} )$ to
    $L^1( {X_1\times X_2} )$.
\end{lemma}

The proof of Lemma~3.1 follows directly from that of the
one-parameter version: Lemma~4.3 in \cite{HLMMY}. The proof
given there is independent of the number of parameters. We omit
the details here.

\begin{lemma}\label{leAtom}
    Let $a$ be an $(H^1_{L_1,L_2}, 2, N)$-atom with
    $N > \max\{n_1,n_2\}/4$.  Let $S$ denote the area
    function defined in~\eqref{esf}. Then
    \begin{eqnarray}\label{e4.11}
        \|Sa\|_{1}\leq C,
    \end{eqnarray}
    where $C$ is a positive constant independent of $a$.
\end{lemma}

Given Lemma~\ref{leAtom}, we may apply
Lemma~\ref{lemma-of-T-bd-on product Hp} with $T=S$ to obtain
\[
    \|f\|_{H^1_{L_1, L_2}( {X_1\times X_2} )}
    :=\|Sf\|_{L^1({X_1\times X_2})}
    \leq C\|f\|_{\mathbb{H}^1_{L_1,L_2,at,N}({X_1\times X_2} )},
\]
from which Step 1 follows.

To finish Step 1, it therefore suffices to verify estimate
(\ref{e4.11}) in Lemma~\ref{leAtom}. To do so, we apply
Journ\'e's covering lemma.

We recall from~\cite{HLL} the formulation of Journ\'e's Lemma
\cite{J,P} in the setting of spaces of homogeneous type. Let
$(X_i,d_i,\mu_i)$, $i = 1$, 2, be spaces of homogeneous type
and let $\{I_{\alpha_i}^{k_i}\subset X_i\}$, $i = 1$, 2,  be
open cubes as in Lemma~\ref{lemma-dyadic-cubes}. Let $\mu =
\mu_1\times\mu_2$ denote the product measure on~$X_1\times
X_2$. The open set $I_{\alpha_1}^{k_1}\times
I_{\alpha_2}^{k_2}$ for $k_1$, $k_2\in \mathbb{Z}$,
$\alpha_1\in I_{k_1}$ and $\alpha_2\in I_{k_2}$, is called a
\emph{dyadic rectangle in $X_1\times X_2$}. Let $\Omega\subset
X_1\times X_2$ be an open set of finite measure. Denote by
$m(\Omega)$ the maximal dyadic rectangles contained in
$\Omega$, and by $m_{i}(\Omega)$ the family of dyadic
rectangles $R\subset\Omega$ which are maximal in the
$x_i$-direction, for $i = 1$, $2$.

In what follows, we let $R = I\times J$ denote any dyadic
rectangle in $X_1\times X_2$. Given $R = I\times J\in
m_1(\Omega)$, let $\widehat{J}$ be the largest dyadic cube
containing $J$ such that
\[
    \mu\big((I\times \widehat{J}) \cap \Omega\big)
    > \frac{1}{2}\,\mu(I\times \widehat{J}).
\]
%In other words, $\widehat{J}$ is the largest dyadic cube containing $J$
%such that $ I \times \widehat{J} \subset \widetilde{\Omega}$, where
%$\widetilde{\Omega} := \{x\in X_1\times X_2 :\
%M_s(\chi_{\Omega})(x) > 1/2\}$ and $M_s$ denotes the strong
%maximal function.

Similarly, given $R = I\times J\in m_2(\Omega)$, let
$\widehat{I}$ be the largest dyadic cube containing $I$ such
that
\[
    \mu\big((\widehat{I}\times J) \cap \Omega\big)
    > \frac{1}{2} \, \mu(\widehat{I}\times J).
\]
Also, let $w(x)$ be any increasing function such that
$\sum_{j=0}^\infty jw(c2^{-j})<\infty$, where $c$ is a fixed
positive constant. In particular, we may take $w(x) = x^\delta$
for any $\delta > 0$.

\begin{lemma}[\cite{HLL}]\label{lemma-Journe}
    Let $\Omega\subset X_1\times X_2$ be an open set
    with finite measure. Then
    \begin{eqnarray}\label{J2}
        \sum_{R = I\times J\in m_1(\Omega)}
            \mu(R) w\Big({\ell(J)\over\ell(\widehat{J})}\Big)
        \leq C\mu(\Omega)
        \end{eqnarray}
    and
    \begin{eqnarray}\label{J1}
        \sum_{R = I\times J\in m_2(\Omega)}
            \mu(R) w\Big({\ell(I)\over\ell(\widehat{I})}\Big)
        \leq C\mu(\Omega),
    \end{eqnarray}
    for some constant $C$ independent of~$\Omega$.
\end{lemma}

\medskip

\begin{proof}[Proof of Lemma \ref{leAtom}]
Given an $(H_{L_1,L_2}^1,2,N)$-atom $a$, let $\Omega$ be an
open set of finite measure in $X_1\times X_2$ as in
Definition~\ref{def H1 atom} such that $ a = \sum_{R\in
m(\Omega)} a_R $ is supported in~$\Omega$.

For each rectangle $R = I \times J\subset\Omega$, let $ I^* $
be the largest dyadic cube in $X_1$ containing $I$ such that $
I^* \times J \subset \widetilde{\Omega}$, where
$\widetilde{\Omega} := \{x\in X_1\times X_2 :\
M_s(\chi_{\Omega})(x) > 1/2\}$ and $M_s$ denotes the strong
maximal function. Next, let $ J^* $ be the largest dyadic cube
in $X_2$ containing $J$ such that $ I^* \times J^* \subset
\widetilde{\widetilde{\Omega}}$, where
$\widetilde{\widetilde{\Omega}} := \{x\in X_1\times X_2 :\
M_s(\chi_{\widetilde{\Omega}})(x) > 1/2\}$.

Now let $R^*$ be the 100-fold dilate of $ I^* \times  J^* $
concentric with $ I^* \times  J^* $. That is, $R^* =  100I^*
\times 100J^*$ is the product of the balls $100I^*$ and
$100J^*$ centered at the centers of $I^*$ and~$J^*$
respectively, as defined in
Section~\ref{sec:assumptions_main_results}. An application of
the strong maximal function theorem shows that
$\mu\big(\cup_{R\subset\Omega} R^*\big)\leq
C\mu(\widetilde{\widetilde{\Omega}})\leq
C\mu(\widetilde{\Omega})\leq C\mu(\Omega)$.

Then we write
\[
    \|Sa\|_{L^1({X_1\times X_2})}
    = \|Sa\|_{L^1(\cup R^*)} + \|Sa\|_{L^1((\cup R^*)^c)}.
\]
Thus, by H\"older's inequality and the property (iii) of the
$(H_{L_1,L_2,}^1,2,N)$-atom, we see that the first term
on the right-hand side is bounded by
\begin{eqnarray*}%\label{SL alpha uniformly bd on inside of Omega}
    \|Sa\|_{L^1(\cup R^*)}
    \leq \mu(\cup R^*)^{1/2} \|Sa\|_{L^2( {X_1\times X_2} )}
    \leq  C \mu(\Omega)^{1/2}\|a\|_{L^2( {X_1\times X_2} )}
    \leq  C.
\end{eqnarray*}

Now it suffices to prove that
\begin{eqnarray}\label{SL alpha uniformly bd on outside of Omega}
    \int_{(\bigcup R^*)^c}|Sa(x_1,x_2)|\,d\mu_1(x_1)\,d\mu_2(x_2)
    \leq C.
\end{eqnarray}

From the definition of $a$, we see that the left-hand side of
(\ref{SL alpha uniformly bd on outside of Omega}) is controlled
by
\begin{eqnarray}\label{DE}
    &&\sum_{R\in m(\Omega) }
    \int_{(R^*)^c}|Sa_R(x_1,x_2)|\,d\mu_1(x_1)\,d\mu_2(x_2)\\
    &&\leq \sum_{R\in m(\Omega) } \int_{(100 I^* )^c\times
    X_2}|Sa_R(x_1,x_2)|\,d\mu_1(x_1)\,d\mu_2(x_2)\nonumber\\
    &&\hskip1cm+\sum_{R\in m(\Omega) }\int_{X_1\times
    (100 J^* )^c}|Sa_R(x_1,x_2)|\,d\mu_1(x_1)\,d\mu_2(x_2)\nonumber\\
    &&=: D + E.\nonumber
\end{eqnarray}

It suffices to verify that the term $D$ is bounded by a
positive constant $C$ independent of  the atom $a$, since the
estimate for $E$ follows symmetrically. For the term $D$, by
splitting the region of integration $(100 I^* )^c\times X_2 $
into $(100 I^* )^c\times 100J $ and $ (100 I^* )^c\times
(100J)^c$, we  write $D$ as $D^{(a)}+D^{(b)}$.

Let us first estimate the term $D^{(a)}$. Using H\"older's inequality,
we have
\begin{eqnarray}\label{estimate D1} \ \ \
    D^{(a)}
    &\leq& C\sum_{R\in m(\Omega) } \mu_2(J)^{{1/2}}
        \int_{(100 I^* )^c }\Big( \int_{100J}|Sa_R(x_1,x_2)|^2\,d\mu_2(x_2)\Big)^{1/2}\,d\mu_1(x_1).
\end{eqnarray}
Next, we claim that
\begin{eqnarray}\label{claim D1}
    &&\int_{( 100I^* )^c}\Big(\int_{100J}|Sa_R(x_1,x_2)|^2\,d\mu_2(x_2)
        \Big)^{1/ 2}\,d\mu_1(x_1)\\
    &&\leq C \Big({\ell(I)\over\ell(I^*)}\Big)^{\epsilon_1} \mu_1(I)^{{1/2}}
        \Big(\ell(I)^{-4N}\ell(J)^{-4N}
        \|(\mathbbm{1}_1\otimes (\ell(J)^2L)^N)b_{R}\|^2_{L^2({X_1\times X_2})}
        \Big)^{1/ 2}\nonumber
\end{eqnarray}
for some $\epsilon_1>0$. Assuming this claim holds, then by
using H\"older's inequality, Journ\'e's Lemma and
property~(2)(iii) of Definition~\ref{def H1 atom}, we have
\begin{eqnarray*}
    D^{(a)}&\leq& C \Big(\sum_{R\in m(\Omega)}
        \mu(R)\Big({\ell(I)\over\ell(I^*)}\Big)^{2\epsilon_1} \Big)^{{1\over 2}}\\
    &&\ \times\Big(\sum_{R\in m(\Omega) } \ell(I)^{-4N}\ell(J)^{-4N}
        \|(\mathbbm{1}_1\otimes (\ell(J)^2L)^N)b_{R}\|^2_{L^2({X_1\times X_2})}
        \bigg)^{1\over 2}\\[4pt]
    &\leq& C \mu(\Omega)^{{1\over 2}} \mu(\Omega)^{-{1\over 2}} \\[4pt]
    &\leq& C.
\end{eqnarray*}

It remains to verify the claim~\eqref{claim D1}. Set
$a_{R,2}=(\mathbbm{1}_1\otimes L_2^N)b_R$; then $a_R =
(L_1^N\otimes \mathbbm{1}_2)a_{R,2}$. Then, from the definition
of the area function, we have
\begin{eqnarray}\label{e1 for 100J}
    &&\int_{100J}|Sa_R(x_1,x_2)|^2\,d\mu_2(x_2)\\
    %&&\
    &&=\int_{100J}\int_{\Gamma_1(x_1)}\int_{\Gamma_2(x_2)}\Big|\big((t_1^2L_1)^{N+1}e^{-t_1^2L_1}\otimes
    t_2^2L_2e^{-t_2^2L_2}\big)(a_{R,2})(y_1,y_2)\Big|^2\nonumber\\
    &&\hskip3cm {\,d\mu_2(y_2)dt_2\over
    t_2V(x_2,t_2)}{\,d\mu_1(y_1)dt_1\over t_1^{1+4N}V(x_1,t_1)}\,d\mu_2(x_2)\nonumber\\
    && = \int_{\Gamma_1(x_1)}\bigg[
    \int_{100J}\int_{\Gamma_2(x_2)}\nonumber\\
    &&\hskip.5cm\Big|t_2^2L_2e^{-t_2^2L_2}\big((t_1^2L_1)^{N+1}e^{-t_1^2L_1}
    a_{R,2}(y_1,\cdot)\big)(y_2)\Big|^2{\,d\mu_2(y_2)dt_2\over
    t_2V(x_2,t_2)}\,d\mu_2(x_2) \bigg]{\,d\mu_1(y_1)dt_1\over t_1^{1+4N}V(x_1,t_1)}\nonumber\\[4pt]
    && \leq C\int_{\Gamma_1(x_1)}
    \int_{X_2}\Big|(t_1^2L_1)^{N+1}e^{-t_1^2L_1}
    a_{R,2}(y_1,x_2)\Big|^2\,d\mu_2(x_2){\,d\mu_1(y_1)dt_1\over t_1^{1+4N}V(x_1,t_1)},\nonumber
\end{eqnarray}
where the last inequality follows from the $L^2$ estimate of the area function on $X_2$.

Define $U_j(I)=2^{j}I\backslash 2^{j-1}I$ for $j\geq1$. Then we
see that $(100I^*)^c \subset \cup_{j>4}U_j(I)$. Moreover, we
have that $\mu_1(U_j(I))\leq C2^{jn_1}\mu_1(I)$. Then, by
H\"older's inequality and the estimate in (\ref{e1 for 100J}),
we get
\begin{eqnarray*}
    &&\int_{(100I^*)^c}\bigg( \int_{100J}|Sa_R(x_1,x_2)|^2\,d\mu_2(x_2)
    \bigg)^{1\over 2}\,d\mu_1(x_1)\\[4pt]
    &&\ \leq C\sum_{j>4} \mu_1(U_j(I))^{1/2}\mu_1(I)^{{1\over 2}}
    \bigg(\int_{( 100I^* )^c\bigcap U_j(I)}  \int_0^\infty\int_{d_1(x_1,y_1)<t_1}
    \int_{X_2}\\
    &&\hskip2cm \Big|(t_1^2L_1)^{N+1}e^{-t_1^2L_1}
    a_{R,2}(y_1,x_2)\Big|^2\,d\mu_2(x_2){\,d\mu_1(y_1)dt_1\over
    t_1^{1+4N}V(x_1,t_1)} \,d\mu_1(x_1) \bigg)^{1\over 2}.
\end{eqnarray*}
Next, we split the integral area $(0,\infty)$ for $t_1$ into
three parts: $(0,\ell(I))$, $(\ell(I), d_1(x_1,x_I)/4)$ and
$(d_1(x_1,x_I)/4,\infty)$. Then the right-hand side of the
above inequality is bounded by the sum of the following three
terms
$$ D_{1}^{(a)}+D_{2}^{(a)}+D_{3}^{(a)}, $$
where
\begin{eqnarray*}
    D_{1}^{(a)}
    &:=&C\sum_{j>4} 2^{{jn_1/ 2}}\mu_1(I)^{{1\over 2}}
        \bigg\{\int_{X_2}\int_{( 100I^* )^c\bigcap U_j(I)}
        \int_0^{\ell(I)}\int_{d_1(x_1,y_1)<t_1}\\[4pt]
    &&\hskip.6cm  \times \Big|(t_1^2L_1)^{N+1}e^{-t_1^2L_1}
        a_{R,2}(y_1,x_2)\Big|^2{\,d\mu_1(y_1)dt_1\over t_1^{1+4N}V(x_1,t_1)}
        \,d\mu_1(x_1) \,d\mu_2(x_2)\bigg\}^{1\over 2},
\end{eqnarray*}
and $D_{2}^{(a)}$ and $D_{3}^{(a)}$ are the same as
$D_{1}^{(a)}$ with the integral $\int_0^{\ell(I)}$ replaced by
$\int_{\ell(I)}^{d_1(x_1,x_I)/4}$ and
$\int_{d_1(x_1,x_I)/4}^\infty$, respectively. Here we use $x_I$
to denote the center of the dyadic cube $I$.

We first consider the term $D_{1}^{(a)}$. We define $E_j(I):=\{y_1:
d_1(x_1,y_1)<\ell(I)\ {\rm for\ some\ } x_1\in ( 100I^* )^c\cap
U_j(I)\}$. Then we can see that ${\rm
dist}(E_j(I),I)>2^{j-2}\ell(I)+\ell(I^*)$. Now we have
\begin{align*}
    D_{1}^{(a)}
    &\leq C\sum_{j>4} 2^{{jn_1/ 2}}\mu_1(I)^{{1\over 2}}
        \bigg\{\int_{X_2}\int_0^{\ell(I)}\int_{E_j(I)}
        \Big|(t_1^2L_1)^{N+1}e^{-t_1^2L_1}
        \alpha_{R,2}(y_1,x_2)\Big|^2{\,d\mu_1(y_1)dt_1\over t_1^{1+4N}}
        \,d\mu_2(x_2)\bigg\}^{1\over 2}\\
    &\leq C\sum_{j>4} 2^{{jn_1/ 2}}\mu_1(I)^{{1\over 2}}
        \bigg\{\int_0^{\ell(I)}e^{-(2^{j-2}\ell(I)+\ell(I^*))^2/
        (ct_1^2)}{dt_1\over t_1^{1+4N}}
        \ \|a_{R,2}\|_{L^2({X_1\times X_2})}^2\bigg\}^{1\over 2}\\
    &\leq C\sum_{j>4} 2^{{jn_1/ 2}}\mu_1(I)^{{1\over 2}} \bigg\{
        {\ell(I)^{\beta} \over (2^{j-2}\ell(I)+\ell(I^*))^{\beta}}\
        \ell(I)^{-4N}\ \|a_{R,2}\|_{L^2({X_1\times X_2})}^2\bigg\}^{1\over
        2},\hskip.2cm
\end{align*}
where the second inequality follows from the Davies--Gaffney
estimates, and the third inequality follows from the fact that
$e^{-x}\leq x^{-\beta}$ for all $x>0$ and $\beta>0$ and that we
choose $\beta$ satisfying $\beta>4N$.

Moreover, noting that
\begin{eqnarray}\label{e1 for claim D1}
    \sum_{j>4} 2^{{jn_1/ 2}} {\ell(I)^{\beta/2} \over
        (2^{j-2}\ell(I)+\ell(I^*))^{\beta/2}}
    \leq \Big({\ell(I)\over\ell(I^*)}\Big)^{n_1/2-\beta/2},
\end{eqnarray}
we obtain that $D_{1}^{(a)}$ is bounded by the right-hand side of
(\ref{claim D1}) for $\epsilon_1:=\beta/2-n_1/2$.

Next we consider the term $D_{2}^{(a)}$. Similarly, we set
\[
    F_j(I)
    := \{y_1: d_1(x_1,y_1)<{d_1(x_1,x_I)/ 4}~\text{for some}~x_1
        \in (100I^* )^c\cap U_j(I)\}.
\]
We see that ${\rm dist}(F_j(I),I) > 2^{j-3}\ell(I) +
\ell(I^*)$. Now we have
\begin{align*}
    D_{2}^{(a)} &\leq C\sum_{j>4} 2^{{jn_1/ 2}}\mu_1(I)^{{1\over 2}}
    \bigg\{\int_{X_2}\int_{\ell(I)}^\infty\int_{F_j(I)}
    \Big|(t_1^2L_1)^{N+1}e^{-t_1^2L_1}
    a_{R,2}(y_1,x_2)\Big|^2{\,d\mu_1(y_1)dt_1\over t_1^{1+4N}}
    \,d\mu_2(x_2)\bigg\}^{1\over2}\\
     &\leq C\sum_{j>4} 2^{{jn_1/ 2}}\mu_1(I)^{{1\over 2}}
    \bigg\{\int_{\ell(I)}^{\infty}e^{-(2^{j-3}\ell(I)+\ell(I^*))^2/
    (ct_1^2)}{dt_1\over t_1^{1+4N}}
    \ \|a_{R,2}\|_{L^2({X_1\times X_2})}^2\bigg\}^{1\over2}\\[4pt]
     &\leq C\sum_{j>4} 2^{{jn_1/ 2}}\mu_1(I)^{{1\over 2}} \bigg\{
    {\ell(I)^{\beta} \over (2^{j-3}\ell(I)+\ell(I^*))^{\beta}}\
    \ell(I)^{-4N}\ \|a_{R,2}\|_{L^2({X_1\times X_2})}^2\bigg\}^{1\over2},
\end{align*}
where the second inequality follows from the Davies--Gaffney
estimates, and $\beta$ is chosen to satisfy $n_1 < \beta<4N$.
Now using (\ref{e1 for claim D1}), we obtain that $D_{2}^{(a)}$
is bounded by the right-hand side of (\ref{claim D1}) for
$\epsilon_1:=\beta/2-n_1/2$.

Now we turn to the term $D_{3}^{(a)}$. Since $x_1\in ( 100I^*
)^c\cap U_j(I)$, we can see that
$d(x_1,x_I)>2^{j-1}\ell(I)+\ell(I^*)$. Thus, the
Davies--Gaffney estimates imply that
\begin{eqnarray*}
    D_{3}^{(a)} &\leq& C\sum_{j>4} 2^{{jn_1/ 2}}\mu_1(I)^{{1\over 2}}\\
    &&\times
    \bigg\{\int_{X_2}\int_{2^{j-1}\ell(I)+\ell(I^*)}^\infty\int_{X_1}
    \Big|(t_1^2L_1)^{N+1}e^{-t_1^2L_1}
    a_{R,2}(y_1,x_2)\Big|^2{\,d\mu_1(y_1)dt_1\over t_1^{1+4N}} \,d\mu_2(x_2)\bigg\}^{1\over2}\\[4pt]
     &\leq& C\sum_{j>4} 2^{{jn_1/ 2}}\mu_1(I)^{{1\over 2}}
    \bigg\{\int_{2^{j-1}\ell(I)+\ell(I^*)}^\infty {dt_1\over
    t_1^{1+4N}}\ \|a_{R,2}\|_{L^2({X_1\times X_2})}^2\bigg\}^{1\over2}\\[4pt]
    &\leq& C\sum_{j>4} 2^{{jn_1/ 2}}\mu_1(I)^{{1\over 2}}
    \bigg\{{\ell(I)^{4N} \over (2^{j-1}\ell(I)+\ell(I^*))^{4N}}\
    \ell(I)^{-4N}\ \|a_{R,2}\|_{L^2({X_1\times X_2})}^2\bigg\}^{1\over2},
\end{eqnarray*}
 Now using
(\ref{e1 for claim D1}), we obtain that $D_{3}^{(a)}$ is bounded by the
right-hand side of (\ref{claim D1}) for $\epsilon_1:=2N-n_1/2$.

Combining the estimates of $D_{1}^{(a)}$, $D_{2}^{(a)}$ and
$D_{3}^{(a)}$, we obtain that the claim (\ref{claim D1}) holds
for $\epsilon_1:=\beta/2-n_1/2$, and hence $D^{(a)}$ is
uniformly bounded.

We now consider the term $D^{(b)}$. Similar to the estimates
for the term $D^{(a)}$, we set $U_{j_1}(I)=2^{j_1}I\backslash
2^{j_1-1}I $ for $j_1\geq1$ and $U_{j_2}(J)=2^{j_2}J\backslash
2^{j_2-1}J $ for $j_2\geq1$. Then we have $(100I^*)^c \subset
\cup_{j_1>4}U_{j_1}(I)$ and $(100J)^c \subset
\cup_{j_2>4}U_{j_2}(J)$. Moreover, we have the following
measure estimate for the annuli: $\mu_1(U_{j_1}(I))\leq
C2^{j_1n_1}\mu_1(I)$ and $\mu_2(U_{j_2}(J))\leq
C2^{j_2n_2}\mu_2(J)$. Now we have
\begin{eqnarray}\label{term D2}
    \ \ D^{(b)}
    &=& \sum_{R\in m(\Omega)} \int_{( 100I^* )^c}\int_{(100J)^c}
        |Sa_R(x_1,x_2)|\,d\mu_1(x_1)d\mu_2(x_2)\\[4pt]
    &\leq & \sum_{R\in m(\Omega)} \sum_{j_1>4}\sum_{j_2>4}
        \int_{( 100I^* )^c\cap U_{j_1}(I)}\int_{(100S)^c\cap U_{j_2}(J)}
        |Sa_R(x_1,x_2)|\,d\mu_1(x_1)d\mu_2(x_2)\nonumber\\[4pt]
    %&\leq& C \sum_{R\in m(\Omega)} \sum_{j_1>4}\sum_{j_2>4}
    %\mu_1(U_{j_1}(I))^{{1\over 2}}\mu_2(U_{j_2}(J))^{{1\over 2}}\nonumber\\
    %&&\times \bigg(\int_{( 100I^* )^c\bigcap U_{j_1}(I)}\int_{(100J)^c\bigcap
    %U_{j_2}(J)} |Sa_R(x_1,x_2)|^2\,d\mu_1(x_1)d\mu_2(x_2)\bigg)^{1\over 2}\nonumber\\
    &\leq& C \sum_{R\in m(\Omega)} \mu(R)^{1/2}\sum_{j_1>4}\sum_{j_2>4}
        2^{{j_1n_1/ 2}}2^{{j_2n_2/ 2}}\nonumber\\
    &&\times \bigg(\int_{( 100I^* )^c\cap U_{j_1}(I)}\int_{(100J)^c\cap
        U_{j_2}(J)} |Sa_R(x_1,x_2)|^2\,d\mu_1(x_1)d\mu_2(x_2)\bigg)^{1\over 2},\nonumber
\end{eqnarray}
where the second inequality follows from H\"older's inequality.

We claim that
\begin{eqnarray}\label{claim for D2}
    && \sum_{j_1>4}\sum_{j_2>4}
        2^{{j_1n_1/ 2}}2^{{j_2n_2/ 2}}\bigg(\int_{( 100I^* )^c\bigcap U_{j_1}(I)}\int_{(100J)^c\bigcap
        U_{j_2}(J)} |Sa_R(x_1,x_2)|^2\,d\mu_1(x_1)d\mu_2(x_2)\bigg)^{1\over 2}\\
    &&\leq C\Big({\ell(I)\over\ell(I^*)}\Big)^{\epsilon_1}
        \big(\ell(I)^{-4N}\ell(J)^{-4N}\|b_R\|_{L^2(X_1\times X_2)}^2\big)^{1/2} \nonumber
\end{eqnarray}
for some $\epsilon_1>0$, which, together with (\ref{term D2}), implies that
\begin{eqnarray*}
D^{(b)}
    &\leq& C \sum_{R\in m(\Omega)} \mu(R)^{1/2}\Big({\ell(I)\over\ell(I^*)}\Big)^{\epsilon_1}
        \big(\ell(I)^{-4N}\ell(J)^{-4N}\|b_R\|_{L^2(X_1\times X_2)^2}\big)^{1/2}\\
    &\leq& C \Big(\sum_{R\in m(\Omega)} \mu(R)\Big({\ell(I)\over\ell(I^*)}\Big)^{2\epsilon_1}\Big)^{1/2}
        \Big(\sum_{R\in m(\Omega)} \ell(I)^{-4N}\ell(J)^{-4N}\|b_R\|_{L^2(X_1\times X_2)^2}\Big)^{1/2}\\
    &\leq& C \mu(\Omega)^{1/2}\mu(\Omega)^{-1/2}\\
    &\leq& C.
\end{eqnarray*}

\noindent From the definitions of the area function $Sf$ and
the $(H_{L_1,L_2}^1,2,N)$-atom $a_R$, we have
\begin{eqnarray*}
    |Sa_R(x)|^2
    &=&\int_0^\infty\int_{d_1(x_1,y_1)<t_1}\int_0^\infty\int_{d_2(x_2,y_2)<t_2}\Big|(t_1^2L_1)^{N+1}e^{-t_1^2L_1}\otimes
        (t_2^2L_2)^{N+1}e^{-t_2^2L_2}(b_R)(y_1,y_2)\Big|^2\\[4pt]
    &&\hskip.7cm\times {\,d\mu_1(y_1)dt_1\over t_1^{1+4N}V(x_1,t_1)}{\,d\mu_2(y_2)dt_2\over
        t_2^{1+4N}V(x_2,t_2)}.
\end{eqnarray*}
Similarly to the estimate for the term  $D^{(a)}$, we split the
region of integration $(0,\infty)$ for $t_1$ into three parts
$(0,\ell(I))$, $(\ell(I), d_1(x_1,x_I)/4)$ and
$(d_1(x_1,x_I)/4,\infty)$, and the region of integration
$(0,\infty)$ for $t_2$ into three parts $(0,\ell(J))$,
$(\ell(J), d_2(x_2,x_J)/4)$ and $(d_2(x_2,x_J)/4,\infty)$.
Hence $|Sa_R(x)|^2$ is decomposed into
\begin{eqnarray*}
    &&|Sa_R(x)|^2\\[5pt]
    &&\ =\bigg( \int_0^{\ell(I)}\!\!\int_0^{\ell(J)} +
    \int_0^{\ell(I)}\!\!\int_{\ell(J)}^{d_2(x_2,x_J)\over 4}+
    \int_0^{\ell(I)}\!\!\int_{d_2(x_2,x_J)\over
    4}^\infty+\int_{\ell(I)}^{d_1(x_1,x_I)\over 4}\!\!\int_0^{\ell(J)}+
    \int_{\ell(I)}^{d_1(x_1,x_I)\over
    4}\!\!\int_{\ell(J)}^{d_2(x_2,x_J)\over 4} \\[4pt]
    &&\hskip1cm   +
    \int_{\ell(I)}^{d_1(x_1,x_I)\over 4}\int_{d_2(x_2,x_J)\over 4}^\infty
    %\\[4pt]
    %&&\hskip.7cm
    +\int_{d_1(x_1,x_I)\over 4}^\infty\int_0^{\ell(J)} +
    \int_{d_1(x_1,x_I)\over 4}^\infty\int_{\ell(J)}^{d_2(x_2,x_J)\over 4} +
    \int_{d_1(x_1,x_I)\over 4}^\infty\int_{d_2(x_2,x_J)\over 4}^\infty
    \bigg)\\[4pt]
    &&\hskip1.7cm
    \int_{d_1(x_1,y_1)<t_1}\int_{d_2(x_2,y_2)<t_2}\Big|(t_1^2L_1)^{N+1}e^{-t_1^2L_1}\otimes
    (t_2^2L_2)^{N+1}e^{-t_2^2L_2}(b_R)(y_1,y_2)\Big|^2\\[4pt]
    &&\hskip2.7cm {\,d\mu_1(y_1)dt_1\over t_1^{1+4N}V(x_1,t_1)}{\,d\mu_2(y_2)dt_2\over
    t_2^{1+4N}V(x_2,t_2)}\\
    &&=:\mathbf{d}_{1,1}(x_1,x_2)+\mathbf{d}_{1,2}(x_1,x_2)+\mathbf{d}_{1,3}(x_1,x_2)+\mathbf{d}_{2,1}(x_1,x_2)+
    \mathbf{d}_{2,2}(x_1,x_2)\\
    &&\hskip1cm+\mathbf{d}_{2,3}(x_1,x_2)+\mathbf{d}_{3,1}(x_1,x_2)+\mathbf{d}_{3,2}(x_1,x_2)+\mathbf{d}_{3,3}(x_1,x_2)\\
    &&= \sum_{\iota=1}^3\sum_{\kappa=1}^3 \mathbf{d}_{\iota,\kappa}(x_1,x_2).
\end{eqnarray*}

Now for $\iota=1,2,3$ and $\kappa=1,2,3$ we set
\begin{eqnarray*}
    D_{\iota,\kappa}^{(b)}:=C  \sum_{j_1>4}\sum_{j_2>4}
    2^{{j_1n_1\over 2}}2^{{j_2n_2\over 2}}\bigg(\int_{( 100I^* )^c\bigcap U_{j_1}(I)}\int_{(100J)^c\bigcap
    U_{j_2}(J)} \mathbf{d}_{\iota,\kappa}(x_1,x_2) \,d\mu_1(x_1)d\mu_2(x_2)\bigg)^{1\over 2}.
\end{eqnarray*}

We first consider $D_{1,1}^{(b)}$. Similar to the estimate in $D_{1}^{(a)}$, we define $E_{j_1}(I):=\{y_1:
d_1(x_1,y_1)<\ell(I)\ {\rm for\ some\ } x_1\in ( 100I^* )^c\bigcap
U_{j_1}(I)\}$ and $E_{j_2}(J):=\{y_2:
d_2(x_2,y_2)<\ell(J)\ {\rm for\ some\ } x_2\in ( 100J )^c\bigcap
U_{j_2}(J)\}$. Then we get ${\rm
dist}(E_{j_1}(I),I)>2^{j_1-2}\ell(I)+\ell(I^*)$ and ${\rm
dist}(E_{j_2}(J),J)>2^{j_1-2}\ell(J)$. Now we have
\begin{eqnarray*}
    &&\int_{( 100I^* )^c\bigcap U_{j_1}(I)}\int_{(100J)^c\bigcap
    U_{j_2}(J)} \mathbf{d}_{1,1}(x_1,x_2) \,d\mu_1(x_1)d\mu_2(x_2)\\
    &&=\int_0^{\ell(I)}\int_{E_{j_1}(I)}\int_0^{\ell(J)}\int_{E_{j_2}(J)}  \Big|(t_1^2L_1)^{N+1}e^{-t_1^2L_1}\otimes
    (t_2^2L_2)^{N+1}e^{-t_2^2L_2}(b_R)(y_1,y_2)\Big|^2\\
    &&\hskip2cm{\,d\mu_1(y_1)dt_1\over t_1^{1+4N}}{\,d\mu_2(y_2)dt_2\over
    t_2^{1+4N}}\\
    &&\leq C \int_0^{\ell(I)}e^{-(2^{j_1-2}\ell(I)+\ell(I^*))^2/
    (ct_1^2)}{dt_1\over t_1^{1+4N}}\int_0^{\ell(J)}e^{-(2^{j_2-2}\ell(J))^2/
    (ct_2^2)}{dt_2\over t_2^{1+4N}}
    \ \|b_{R}\|_{L^2({X_1\times X_2})}^2\\
    &&\leq C{\ell(I)^{\beta} \over (2^{j_1-2}\ell(I)+\ell(I^*))^{\beta}}\
    \ell(I)^{-4N} 2^{-j_2\beta}\ell(J)^{-4N}\|b_{R}\|_{L^2({X_1\times X_2})}^2,
\end{eqnarray*}
where the second inequality follows from the Davies--Gaffney
estimates, and the third inequality follows from the fact that
$e^{-x}\leq x^{-\beta}$ for all $x>0$ and $\beta>0$ and that we
choose $\beta$ satisfying $\beta>4N$.

Thus,
\begin{eqnarray*}
    D_{1,1}^{(b)}
    &\leq& C  \sum_{j_1>4}
        2^{{j_1n_1\over 2}}{\ell(I)^{\beta\over2} \over (2^{j_1-2}\ell(I)+\ell(I^*))^{\beta\over2}}
        \sum_{j_2>4}2^{{j_2n_2\over 2}}2^{-j_2\beta\over2}\big(\ell(I)^{-4N}\ell(J)^{-4N}\|b_R\|_{L^2(X_1\times X_2)^2}\big)^{1/2}\\
    &\leq&  C\Big({\ell(I)\over\ell(I^*)}\Big)^{\epsilon_1}\big(\ell(I)^{-4N}\ell(J)^{-4N}\|b_R\|_{L^2(X_1\times X_2)}^2\big)^{1/2},
\end{eqnarray*}
where the second inequality follows from (\ref{e1 for claim
D1}) with $\epsilon_1 := \beta/2-n_1/2$. Note that
$\beta>\max\{n_1,n_2\}$ follows from the fact that
$N>\max\{n_1/4,n_2/4\}$.

As for $D_{1,2}^{(b)}$, similar to the term $D_{2}^{(a)}$, set
$F_{j_2}(J):=\{y_2: d_2(x_2,y_2)<{d_2(x_2,x_J)/ 4}\ {\rm for\
some\ }\newline x_2\in (100J )^c\bigcap U_{j_2}(J)\}$. Then  we
can see that ${\rm dist}(F_{j_2}(J),J)>2^{j_2-3}\ell(J)$. Now
we have
\begin{eqnarray*}
    &&\int_{( 100I^* )^c\bigcap U_{j_1}(I)}\int_{(100J)^c\bigcap
    U_{j_2}(J)} \mathbf{d}_{1,2}(x_1,x_2) \,d\mu_1(x_1)d\mu_2(x_2)\\
    &&=\int_0^{\ell(I)}\int_{E_{j_1}(I)}\int_{\ell(J)}^{d_2(x_2,x_J)\over4}\int_{F_{j_2}(J)}  \Big|(t_1^2L_1)^{N+1}e^{-t_1^2L_1}\otimes
    (t_2^2L_2)^{N+1}e^{-t_2^2L_2}(b_R)(y_1,y_2)\Big|^2\\
    &&\hskip2cm{\,d\mu_1(y_1)dt_1\over t_1^{1+4N}}{\,d\mu_2(y_2)dt_2\over
    t_2^{1+4N}}\\
    &&\leq C \int_0^{\ell(I)}e^{-(2^{j_1-2}\ell(I)+\ell(I^*))^2/
    (ct_1^2)}{dt_1\over t_1^{1+4N}}\int_{\ell(J)}^\infty e^{-(2^{j_2-2}\ell(J))^2/
    (ct_2^2)}{dt_2\over t_2^{1+4N}}
    \ \|b_{R}\|_{L^2({X_1\times X_2})}^2\\
    &&\leq C{\ell(I)^{\beta_1} \over (2^{j_1-2}\ell(I)+\ell(I^*))^{\beta_1}}\
    \ell(I)^{-4N} 2^{-j_2\beta_2}\ell(J)^{-4N}\|b_{R}\|_{L^2({X_1\times X_2})}^2,
\end{eqnarray*}
where the second inequality follows from the Davies--Gaffney
estimates, and the third inequality follows from the fact that
$e^{-x}\leq x^{-\beta}$ for all $x>0$ and $\beta>0$, and that
we choose $\beta_1$ satisfying $\beta_1>4N$ and $\beta_2$
satisfying $n_2<\beta_2<4N$. Hence, similar to the estimate of
the term $D_{1,1}^{(b)}$,
\begin{eqnarray*}
    D_{1,2}^{(b)}
    &\leq&  C\Big({\ell(I)\over\ell(I^*)}\Big)^{\epsilon_1}\big(\ell(I)^{-4N}\ell(J)^{-4N}\|b_R\|_{L^2(X_1\times X_2)}^2\big)^{1/2}
\end{eqnarray*}
with $\epsilon_1:=\beta_1/2-n_1/2$. Note that $\beta_1>n_1$ follows from the fact that $N>n_1/4$.

As for $D_{1,3}^{(b)}$, since $x_2 \in (100J)^c\cap U_{j_2}(J)
$, we see that $d_2(x_2,x_J)>2^{j_2-1}\ell(J)$. Thus, the
Davies--Gaffney estimates imply that
\begin{eqnarray*}
    &&\int_{( 100I^* )^c\bigcap U_{j_1}(I)}\int_{(100J)^c\bigcap
    U_{j_2}(J)} \mathbf{d}_{1,3}(x_1,x_2) \,d\mu_1(x_1)d\mu_2(x_2)\\
    &&=\int_0^{\ell(I)}\int_{E_{j_1}(I)}\int_{2^{j_2-1}\ell(J)}^{\infty}\int_{X_2}  \Big|(t_1^2L_1)^{N+1}e^{-t_1^2L_1}\otimes
    (t_2^2L_2)^{N+1}e^{-t_2^2L_2}(b_R)(y_1,y_2)\Big|^2\\
    &&\hskip2cm{\,d\mu_1(y_1)dt_1\over t_1^{1+4N}}{\,d\mu_2(y_2)dt_2\over
    t_2^{1+4N}}\\
    &&\leq C \int_0^{\ell(I)}e^{-(2^{j_1-2}\ell(I)+\ell(I^*))^2/
    (ct_1^2)}{dt_1\over t_1^{1+4N}}\int_{2^{j_2-1}\ell(J)}^\infty {dt_2\over t_2^{1+4N}}
    \ \|b_{R}\|_{L^2({X_1\times X_2})}^2\\
    &&\leq C{\ell(I)^{\beta_1} \over (2^{j_1-2}\ell(I)+\ell(I^*))^{\beta_1}}\
    \ell(I)^{-4N} 2^{-4Nj_2}\ell(J)^{-4N}\|b_{R}\|_{L^2({X_1\times X_2})}^2,
\end{eqnarray*}
in which we choose $\beta_1>4N$. Hence, we have
\begin{eqnarray*}
    D_{1,3}^{(b)}
    &\leq&  C\Big({\ell(I)\over\ell(I^*)}\Big)^{\epsilon_1}\big(\ell(I)^{-4N}\ell(J)^{-4N}\|b_R\|_{L^2(X_1\times X_2)}^2\big)^{1/2}
\end{eqnarray*}
with $\epsilon_1:=\beta_1/2-n_1/2$. Note that $\beta_1>n_1$ follows from the fact that $N>n_1/4$.

For the remaining terms $D^{(b)}_{\iota,\kappa}$ for
$\iota=2,3$ and $\kappa=1,2,3$, we estimate the integral with
respect to the first variable $t_1$ in a way similar to that
for $D^{(a)}_\iota$ above, while for the integral with respect
to $t_2$, we use an estimate similar to that used for the $t_2$
integral in $D^{(b)}_{1,\kappa}$ above. This completes the
estimate of $D^{(b)}$, and hence that of $D$.

The estimate for the term $E$ is symmetric to that of $D$.

Combining the estimates of $D$ and $E$, we obtain (\ref{SL
alpha uniformly bd on outside of Omega}), which, together with
the fact that $\|Sa\|_{L^1(\cup R^*)}\leq C$, yields the
estimate~\eqref{e4.11}. Thus Lemma~\ref{leAtom} is proved.
\end{proof}

This completes the proof of Step~1.
\end{proof}

%\medskip

\begin{proof}[Proof of Step 2]
Our goal is to show that every function $f\in H_{L_1,L_2}^1(
{X_1\times X_2} )\cap L^2( {X_1\times X_2} )$ has an
$(H^1_{L_1, L_2}, 2, M)$-atom representation, with appropriate
quantitative control of the coefficients. To this end, we
follow the standard tent space approach, and we are now ready
to establish the atomic decomposition of $H_{L_1,L_2}^1(
{X_1\times X_2} )\cap L^2( {X_1\times X_2} )$.

\begin{prop}\label{prop-product H-SL subset H-at}
    Suppose $M\geq 1$. If $f\in
    H_{L_1,L_2}^1( {X_1\times X_2}  )\cap
    L^2( {X_1\times X_2}  )$, then there exist a
    family of $(H^1_{L_1, L_2}, 2, M)$-atoms $\{a_j\}_{j=0}^\infty$ and a sequence of
    numbers $\{\lambda_j\}_{j=0}^\infty\in \ell^1$ such that $f$ can be
    represented in the form $f=\sum\lambda_ja_j$, with the sum
    converging in $L^2( {X_1\times X_2} )$, and
    $$
        \|f\|_{\mathbb{H}^1_{L_1,L_2,at,N}( {X_1\times X_2} )}
        \leq C\sum_{j=0}^\infty|\lambda_j|
        \leq C\|f\|_{H_{L_1,L_2}( {X_1\times X_2} )},
    $$
    where $C$ is independent of $f$. In particular,
    $$
        H_{L_1,L_2}^1( {X_1\times X_2} )\cap L^2( {X_1\times X_2} )\ \
        \subset\ \mathbb{H}^1_{L_1,L_2,at,M}( {X_1\times X_2} ).
    $$
\end{prop}

\begin{proof}
Let $f\in H_{L_1,L_2}^1( X_1\times X_2 )\cap L^2( {X_1\times
X_2} )$. For each $\ell\in\mathbb{Z}$, define
\begin{eqnarray*}
    \Omega_\ell&:=&\{(x_1,x_2)\in X_1\times X_2: Sf > 2^\ell \},\\
    B_\ell&:=&\Big\{R=I_{\alpha_1}^{k_1}\times I_{\alpha_2}^{k_2}:
        \mu( R\cap \Omega_\ell)>{1\over 2A_0}\mu(R),\  \mu( R\cap \Omega_{\ell+1})\leq {1\over 2A_0}\mu(R) \Big\}, {\rm\ and}\\
    \widetilde{\Omega}_\ell&:=&\Big\{(x_1,x_2)\in X_1\times X_2: \mathcal{M}_s(\chi_{\Omega_\ell})>{1\over2A_0} \Big\},
\end{eqnarray*}
where $\mathcal{M}_s$ is the strong maximal function on $
X_1\times X_2$.

For each rectangle $R = I_{\alpha_1}^{k_1}\times
I_{\alpha_2}^{k_2}$ in $X_1\times X_2$, the \emph{tent $T(R)$}
is defined as
$$
    T(R)
    := \big\{ (y_1,y_2,t_1,t_2):\ (y_1,y_2)\in R, t_1\in ( 2^{-k_1},2^{-k_1+1} ], t_2\in ( 2^{-k_2},2^{-k_2+1} ]\big\}.
$$
For brevity, in what follows we will write $\chi_{T(R)}$ for
$\chi_{T(R)}(y_1, y_2, t_1, t_2)$.

Using the reproducing formula, we can write
\begin{align}\label{e2 in section 5.3.3}
    \ \ \ \ f(x_1,x_2)
    &= \int_0^\infty\!\!\int_0^\infty
        \psi(t_1\sqrt{L_1})\psi(t_2\sqrt{L_2})(t_1^2L_1e^{-t_1^2L_1}\otimes t_2^2L_2e^{-t_2^2L_2})(f)(x_1,x_2){dt_1dt_2\over t_1t_2}\\
    &=\int_0^\infty\!\!\int_0^\infty\!\! \int_{X_1}\int_{X_2}
        K_{\psi(t_1\sqrt{L_1})}(x_1,y_1)K_{\psi(t_2\sqrt{L_2})}(x_2,y_2)\nonumber\\
    &\hskip1cm(t_1^2L_1e^{-t_1^2L_1}\otimes t_2^2L_2e^{-t_2^2L_2})(f)(y_1,y_2)d\mu_1(y_1)d\mu_2(y_2){dt_1dt_2\over t_1t_2}\nonumber\\
    &= \sum_{\ell\in\mathbb{Z}}\sum_{R\in B_\ell}  \int_{T(R)} K_{\psi(t_1\sqrt{L_1})}(x_1,y_1)K_{\psi(t_2\sqrt{L_2})}(x_2,y_2)\nonumber\\
    &\hskip1cm(t_1^2L_1e^{-t_1^2L_1}\otimes t_2^2L_2e^{-t_2^2L_2})(f)(y_1,y_2)d\mu_1(y_1)d\mu_2(y_2){dt_1dt_2\over t_1t_2}\nonumber\\
    &=: \sum_{\ell\in\mathbb{Z}}\lambda_\ell a_\ell(x_1,x_2). \nonumber
\end{align}
Here the coefficients $\lambda_\ell$  are defined by
$$
    \lambda_\ell
    := C\bigg\|\bigg( \sum_{R\in B_\ell}  \int_{0}^\infty\!\!\int_{0}^\infty
        \big|(t_1^2L_1e^{-t_1^2L_1}\otimes t_2^2L_2e^{-t_2^2L_2})(f)(y_1,y_2)\big|^2\chi_{T(R)}
        {dt_1dt_2\over t_1t_2}\bigg)^{1/2}\bigg\|_{L^2}\mu(\widetilde{\Omega}_\ell)^{1/2},
$$
Also the functions $a_\ell(x_1,x_2)$ are defined by
\begin{align*}
    a_\ell(x_1,x_2)
    &:={1\over\lambda_\ell}\sum_{R\in B_\ell}  \int_{T(R)} K_{\psi(t_1\sqrt{L_1})}(x_1,y_1)K_{\psi(t_2\sqrt{L_2})}(x_2,y_2)\nonumber\\
    &\hskip1cm(t_1^2L_1e^{-t_1^2L_1}\otimes t_2^2L_2e^{-t_2^2L_2})(f)(y_1,y_2)d\mu_1(y_1)d\mu_2(y_2){dt_1dt_2\over t_1t_2}.
\end{align*}

First, it is easy to verify property (1) in Definition~\ref{def
H1 atom}, since from Lemma~\ref{lemma finite speed} and the
definition of the sets $B_\ell$ and $\widetilde{\Omega}_\ell$,
we obtain that $a_\ell(x_1,x_2)$ is supported in
$\widetilde{\Omega}_\ell$.

Next, we can further write
\begin{eqnarray*}
    a_\ell(x_1,x_2)
    &=& \sum_{\overline{R}\in m(\widetilde{\Omega}_\ell)} a_{\overline{R}}(x_1,x_2),
\end{eqnarray*}
where
\begin{eqnarray*}
    a_{\overline{R}}
    &:=& \sum_{R\in B_\ell, R\subset \overline{R}}{1\over\lambda_\ell}\int_{T(R)} K_{\psi(t_1\sqrt{L_1})}(x_1,y_1)K_{\psi(t_2\sqrt{L_2})}(x_2,y_2)\nonumber\\
    && \hskip1cm(t_1^2L_1e^{-t_1^2L_1}\otimes t_2^2L_2e^{-t_2^2L_2})(f)(y_1,y_2)\,d\mu_1(y_1)d\mu_2(y_2){dt_1dt_2\over t_1t_2}.
\end{eqnarray*}
Then property (i) of (2) in Definition \ref{def H1 atom} holds,
since $a_{\overline{R}}$ can be further written as
$$a_{\overline{R}}= (L_1^N\otimes L_2^N)b_{\overline{R}},  $$ where
\begin{eqnarray*}
    b_{\overline{R}}
    &:=& \sum_{R\in B_\ell, R\subset \overline{R}}{1\over\lambda_\ell}\int_{T(R)}
        t_1^{2N}t_2^{2N}K_{\phi(t_1\sqrt{L_1})}(x_1,y_1)K_{\phi(t_2\sqrt{L_2})}(x_2,y_2)\nonumber\\
    && \hskip1cm(t_1^2L_1e^{-t_1^2L_1}\otimes t_2^2L_2e^{-t_2^2L_2})(f)(y_1,y_2)\,d\mu_1(y_1)d\mu_2(y_2){dt_1dt_2\over t_1t_2}.
\end{eqnarray*}
Next, from Lemma \ref{lemma finite speed}, we obtain that
property (ii) of (2) in Definition \ref{def H1 atom} holds.

We now verify property (iii) of (2). To do so, we write
\[
    \|a_\ell\|_{L^2(X_1\times X_2)}
    = \sup_{h: \|h\|_{L^2(X_1\times X_2)}=1} |\langle a_\ell,h\rangle|.
\]
Then from the definition of $a_\ell$, we have
\begin{eqnarray*}
    \lefteqn{|\langle a_\ell,h\rangle|}\\
    &=& \bigg|\int_{X_1\times X_2} {1\over\lambda_\ell}\sum_{R\in B_\ell}  \int_{T(R)} K_{\psi(t_1\sqrt{L_1})}(x_1,y_1)K_{\psi(t_2\sqrt{L_2})}(x_2,y_2)\nonumber\\
    &&(t_1^2L_1e^{-t_1^2L_1}\otimes t_2^2L_2e^{-t_2^2L_2})(f)(y_1,y_2)d\mu_1(y_1)d\mu_2(y_2){dt_1dt_2\over t_1t_2}\ h(x_1,x_2) d\mu_1(x_1)d\mu_2(x_2) \bigg|\\
    &\leq & {1\over\lambda_\ell}\sum_{R\in B_\ell}  \int_{T(R)} |\psi(t_1\sqrt{L_1})\psi(t_2\sqrt{L_2})(h)(y_1,y_2)|\nonumber\\
    &&\big|(t_1^2L_1e^{-t_1^2L_1}\otimes t_2^2L_2e^{-t_2^2L_2})(f)(y_1,y_2)\big|d\mu_1(y_1)d\mu_2(y_2){dt_1dt_2\over t_1t_2}\\
    &\leq & {1\over\lambda_\ell} \int_{X_1\times X_2}\bigg( \sum_{R\in B_\ell}
        \int_{0}^\infty\!\!\int_{0}^\infty |\psi(t_1\sqrt{L_1})\psi(t_2\sqrt{L_2})(h)(y_1,y_2)|^2\chi_{T(R)}{dt_1dt_2\over t_1t_2} \bigg)^{1/2}\nonumber\\
    &&\bigg( \sum_{R\in B_\ell}  \int_{0}^\infty\!\!\int_{0}^\infty
        \big|(t_1^2L_1e^{-t_1^2L_1}\otimes t_2^2L_2e^{-t_2^2L_2})(f)(y_1,y_2)\big|^2\chi_{T(R)}  {dt_1dt_2\over t_1t_2}\bigg)^{1/2}d\mu_1(y_1)d\mu_2(y_2)\\
    &\leq& {C\over\lambda_\ell}\|h\|_{L^2}\bigg\|\bigg( \sum_{R\in B_\ell}
        \int_{0}^\infty\!\!\int_{0}^\infty \big|(t_1^2L_1e^{-t_1^2L_1}\otimes t_2^2L_2e^{-t_2^2L_2})(f)(y_1,y_2)\big|^2\chi_{T(R)}  {dt_1dt_2\over t_1t_2}\bigg)^{1/2}\bigg\|_{L^2}\\
    &\leq& \mu(\widetilde{\Omega}_\ell)^{-1/2}.
\end{eqnarray*}
In the last inequality, we have used the definition
of~$\lambda_\ell$.

Similarly, from the definition of the function
$b_{\overline{R}}$, we have for each $\sigma_1$, $\sigma_2 \in
\{0, 1, \ldots, N\}$ that
\begin{eqnarray*}
    \lefteqn{\ell(\overline{I})^{-2N}\ell(\overline{J})^{-2N}\|(\ell(\overline{I})^2L_1)^{\sigma_1}\otimes (\ell(\overline{J})^2L_2)^{\sigma_2} b_{\overline{R}}\|_{L^2}}\\
    &&\hskip.2cm=\sup_{h: \|h\|_{L^2}=1}
        \big|\langle \ell(\overline{I})^{-2N}\ell(\overline{J})^{-2N}(\ell(\overline{I})^2L_1)^{\sigma_1}\otimes (\ell(\overline{J})^2L_2)^{\sigma_2} b_{\overline{R}} ,h\rangle\big|\\
    &&\hskip.2cm\leq \sup_{h: \|h\|_{L^2}=1} {C\over\lambda_\ell}\sum_{R\in B_\ell, R\subset \overline{R}}
        \int_{T(R)} |(\ell(\overline{I})^2L_1)^{\sigma_1}\phi(t_1\sqrt{L_1})\otimes(\ell(\overline{J})^2L_2)^{\sigma_2}\phi(t_2\sqrt{L_2})(h)(y_1,y_2)|\nonumber\\
    &&\hskip1cm\big|(t_1^2L_1e^{-t_1^2L_1}\otimes t_2^2L_2e^{-t_2^2L_2})(f)(y_1,y_2)\big|\,d\mu_1(y_1)d\mu_2(y_2){dt_1dt_2\over t_1t_2}.\\
\end{eqnarray*}
As a consequence, using the same approach as in the above
estimates for $a_{\ell}$, we have
\begin{eqnarray*}
    \lefteqn{\sum_{\overline{R}\in m(\widetilde{\Omega}_\ell)}
        \ell(\overline{I})^{-4N}\ell(\overline{J})^{-4N}\|(\ell(\overline{I})^2L_1)^{\sigma_1}\otimes (\ell(\overline{J})^2L_2)^{\sigma_2} b_{\overline{R}}\|_{L^2}^2}\\
    &&\hskip.2cm\leq \sup_{h: \|h\|_{L^2}=1} {C\over\lambda_\ell^2} \sum_{\overline{R}\in m(\widetilde{\Omega}_\ell)}\bigg(\sum_{R\in B_\ell, R\subset \overline{R}}  \int_{T(R)} \\
    &&\hskip1cm
    |(\ell(\overline{I})^2L_1)^{\sigma_1}\phi(t_1\sqrt{L_1})\otimes(\ell(\overline{J})^2L_2)^{\sigma_2}\phi(t_2\sqrt{L_2})(h)(y_1,y_2)|\nonumber\ \ \ \ \  \ \\
    &&\hskip2cm\big|(t_1^2L_1e^{-t_1^2L_1}\otimes t_2^2L_2e^{-t_2^2L_2})(f)(y_1,y_2)\big|\,d\mu_1(y_1)d\mu_2(y_2)\,{dt_1dt_2\over t_1t_2}\bigg)^2\\
    &&\hskip.2cm\leq {C\over\lambda_\ell^2}\bigg\|\bigg( \sum_{R\in B_\ell}
        \int_{0}^\infty\!\!\int_{0}^\infty \big|(t_1^2L_1e^{-t_1^2L_1}\otimes t_2^2L_2e^{-t_2^2L_2})(f)(y_1,y_2)\big|^2\chi_{T(R)}\,  {dt_1dt_2\over t_1t_2}\bigg)^{1/2}\bigg\|_{L^2}^2\\
    &&\hskip.2cm\leq \mu(\widetilde{\Omega}_\ell)^{-1}.
\end{eqnarray*}
The last inequality follows from the definition
of~$\lambda_\ell$.

Combining the above estimate and the estimate for $a_\ell$, we
have established property (iii) of (2) in Definition~\ref{def
H1 atom}. Thus, each $a_\ell$ is an $(H^1_{L_1,L_2},2,N)$-atom.

\smallskip
To see that the atomic decomposition $\sum_\ell \lambda_\ell
a_\ell$ converges to $f$ in the $L^2(X_1\times X_2)$ norm, we
only need to show that $\|\sum_{|\ell|>G} \lambda_\ell
a_\ell\|_{L^2(X_1\times X_2)}\rightarrow 0$ as $G$ tends to
infinity. To see this, first note that
$$
    \Big\|\sum_{|\ell|>G} \lambda_\ell a_\ell\Big\|_{L^2(X_1\times X_2)}
    =\sup_{h:\, \|h\|_{L^2(X_1\times X_2)=1 }}
    \Big|\big\langle \sum_{|\ell|>G} \lambda_\ell a_\ell, h\big\rangle\Big|.
$$
Next, we have
\begin{eqnarray*}
    \lefteqn{\Big|\big\langle \sum_{|\ell|>G} \lambda_\ell a_\ell, h\big\rangle\Big|}\\
    &&\hskip.2cm=\bigg|\int_{X_1\times X_2} \sum_{|\ell|>G} \sum_{R\in B_\ell}  \int_{T(R)} K_{\psi(t_1\sqrt{L_1})}(x_1,y_1)K_{\psi(t_2\sqrt{L_2})}(x_2,y_2)\nonumber\\
    &&\hskip1.2cm(t_1^2L_1e^{-t_1^2L_1}\otimes t_2^2L_2e^{-t_2^2L_2})(f)(y_1,y_2)d\mu_1(y_1)d\mu_2(y_2){dt_1dt_2\over t_1t_2}\ h(x_1,x_2) d\mu_1(x_1)d\mu_2(x_2) \bigg|\\
    &&\hskip.2cm\leq \int_{X_1\times X_2}\bigg( \sum_{|\ell|>G}\sum_{R\in B_\ell}
        \int_{0}^\infty\!\!\int_{0}^\infty |\psi(t_1\sqrt{L_1})\psi(t_2\sqrt{L_2})(h)(y_1,y_2)|^2\chi_{T(R)}{dt_1dt_2\over t_1t_2} \bigg)^{1\over2}\nonumber\\
    &&\hskip.58cm\bigg( \sum_{|\ell|>G}\sum_{R\in B_\ell}\!
        \int_{0}^\infty\!\!\!\int_{0}^\infty \!\!\big|(t_1^2L_1e^{-t_1^2L_1}\otimes t_2^2L_2e^{-t_2^2L_2})(f)(y_1,y_2)\big|^2\chi_{T(R)}  {dt_1dt_2\over t_1t_2}\bigg)^{1\over2}\!d\mu_1(y_1)d\mu_2(y_2)\\
    &&\hskip.2cm\leq C\|h\|_{L^2}\bigg\|\bigg( \sum_{|\ell|>G}\sum_{R\in B_\ell}
        \int_{0}^\infty\!\!\int_{0}^\infty \big|(t_1^2L_1e^{-t_1^2L_1}\otimes t_2^2L_2e^{-t_2^2L_2})(f)(y_1,y_2)\big|^2\chi_{T(R)}  {dt_1dt_2\over t_1t_2}\bigg)^{1\over2}\bigg\|_{L^2}\\
    &&\hskip.2cm\rightarrow 0
\end{eqnarray*}
as $G$ tends to $\infty$, since $\|Sf\|_2 < \infty$.

This implies that $f = \sum_\ell \lambda_\ell a_\ell$ in the
sense of $L^2(X_1\times X_2)$.

\medskip
Next, we verify the estimate for the series
$\sum_\ell|\lambda_\ell|$. To deal with this, we claim that for
each $\ell\in\mathbb{Z}$,
\begin{eqnarray*}
    \sum_{R\in B_\ell}\int_{T(R)} \big|(t_1^2L_1e^{-t_1^2L_1}\otimes t_2^2L_2e^{-t_2^2L_2})(f)(y_1,y_2)\big|^2 d\mu_1(y_1)d\mu_2(y_2){dt_1dt_2\over t_1t_2}
    \leq C2^{2(\ell+1)}\mu(\widetilde{\Omega}_\ell).
\end{eqnarray*}

\noindent First we note that
\begin{equation*}
    \int_{\widetilde{\Omega}_\ell\backslash \Omega_{\ell+1}} (Sf)^2(x_1,x_2)\, d\mu_1(x_1)d\mu_2(x_2)
    \leq  2^{2(\ell+1)}\mu(\widetilde{\Omega}_\ell).
\end{equation*}
Also we point out that
\begin{eqnarray*}
    && \int_{\widetilde{\Omega}_\ell\backslash \Omega_{\ell+1}} (Sf)^2(x_1,x_2)\, d\mu_1(x_1)d\mu_2(x_2)\\
    &&= \int_{\widetilde{\Omega}_\ell\backslash \Omega_{\ell+1}} \int_{\Gamma_1(x_1) }\int_{\Gamma_2(x_2) }\\
    &&\hskip1cm\big|\big(t_1^2L_1e^{-t_1^2L_1}\otimes
    t_2^2L_2e^{-t_2^2L_2}\big)f(y_1,y_2)\big|^2{d\mu_1(y_1)d\mu_2(y_2)\ \! dt_1dt_2\over t_1V(x_1,t_1) t_2V(x_2,t_2)}\, d\mu_1(x_1)d\mu_2(x_2)\\
    &&= \int_0^\infty\!\!\int_0^\infty\!\!\int_{X_1\times X_2}\big|\big(t_1^2L_1e^{-t_1^2L_1}\otimes
    t_2^2L_2e^{-t_2^2L_2}\big)(f)(y_1,y_2)\big|^2\\
    &&\hskip1cm\times\mu(\{(x_1,x_2)\in\widetilde{\Omega}_\ell\backslash \Omega_{\ell+1}:\, d_1(x_1,y_1)<t_1,d_2(x_2,y_2)<t_2\})
        {d\mu_1(y_1)d\mu_2(y_2)\ \! dt_1dt_2\over t_1V(x_1,t_1) t_2V(x_2,t_2)} \\
    &&\geq \sum_{R\in B_\ell}\int_{T(R)}\big|\big(t_1^2L_1e^{-t_1^2L_1}\otimes
    t_2^2L_2e^{-t_2^2L_2}\big)(f)(y_1,y_2)\big|^2\\
    &&\hskip1cm\times\mu(\{(x_1,x_2)\in\widetilde{\Omega}_\ell\backslash \Omega_{\ell+1}:\, d_1(x_1,y_1)<t_1,d_2(x_2,y_2)<t_2\})
        {d\mu_1(y_1)d\mu_2(y_2)\ \! dt_1dt_2\over t_1V(x_1,t_1) t_2V(x_2,t_2)} \\
    &&\geq C\sum_{R\in B_\ell}\int_{T(R)} \big|(t_1^2L_1e^{-t_1^2L_1}\otimes t_2^2L_2e^{-t_2^2L_2})(f)(y_1,y_2)\big|^2
        d\mu_1(y_1)d\mu_2(y_2){dt_1dt_2\over t_1t_2},
\end{eqnarray*}
where the last inequality follows from the definition of $B_\ell$. This shows that the claim holds.

As a consequence, we have
\begin{eqnarray*}
    &&\sum_\ell|\lambda_\ell|\\
    &&\leq C\!\sum_\ell\!\bigg\|\bigg( \sum_{R\in B_\ell}
        \int_{0}^\infty\!\!\int_{0}^\infty \big|(t_1^2L_1e^{-t_1^2L_1}\otimes t_2^2L_2e^{-t_2^2L_2})(f)(y_1,y_2)\big|^2\chi_{T(R)}
        {dt_1dt_2\over t_1t_2}\bigg)^{1/2}\bigg\|_{L^2}\mu(\widetilde{\Omega}_\ell)^{1/2}\\
    &&\leq C\!\sum_\ell\!\bigg(\!\sum_{R\in B_\ell}\int_{T(R)}\! \big|(t_1^2L_1e^{-t_1^2L_1}\otimes t_2^2L_2e^{-t_2^2L_2})(f)(y_1,y_2)\big|^2
        d\mu_1(y_1)d\mu_2(y_2){dt_1dt_2\over t_1t_2}\bigg)^{1/2}\!\!\!\mu(\widetilde{\Omega}_\ell)^{1/2}\\
    &&\leq C\sum_\ell2^{\ell+1}\mu(\widetilde{\Omega}_\ell)\leq C\sum_\ell2^{\ell}\mu(\Omega_\ell)\\
    &&\leq C\|Sf\|_{L^1(X_1\times X_2)}\\
    &&=C\|f\|_{H^1_{L_1,L_2}(X_1\times X_2)}.
\end{eqnarray*}

Therefore,
$$
    \|f\|_{\mathbb{H}_{L_1,L_2,at,N}^1( {X_1\times X_2} )}
    \leq C\|f\|_{H_{L_1,L_2}^1( {X_1\times X_2} )},
$$
which completes the proof of Proposition~\ref{prop-product H-SL
subset H-at}.
\end{proof}

Step 2 is now complete. This concludes the proof of Theorem
\ref{theorem of Hardy space atomic decom}.
\end{proof}

%-----------------------------------------------------------------------
\section{Calder\'on--Zygmund decomposition and interpolation
on $H_{L_1, L_2}^p(X_1\times X_2)$}
\setcounter{equation}{0}
\label{sec:CZ_decomposition_interpolation}

In this section, we provide the proofs of the
Calder\'on--Zygmund decomposition (Theorem~\ref{theorem C-Z
decomposition for Hp}) and the interpolation theorem
(Theorem~\ref{theorem interpolation Hp}) on the Hardy spaces
$H_{L_1\times L_2}^p(X_1\times X_2)$.

%Note that $H^p( X_1\times X_2 )=L^p( X_1\times X_2 )$ for $1<p<\infty$.
%\begin{theorem}\label{theorem interpolation Hp}
%Suppose that $L_1$ and $L_2$ are non-negative self-adjoint operators such that the corresponding
%heat semigroups satisfy Davies--Gaffney estimates.
%Let $T$ be a linear operator which is bounded on $L^2(X_1\times X_2)$ and bounded from $H_{L_1,L_2}^{1}( X_1\times X_2 )$ to $L^{1}( X_1\times X_2 )$. Then $T$ is bounded
%from $H_{L_1,L_2}^p( X_1\times X_2 )$ to $L^p( X_1\times X_2 )$ for all $1<p<2$.
%\end{theorem}
%
%
%
%\begin{theorem}\label{theorem C-Z decomposition for Hp}
%Let $1<p<2,$ $\alpha>0$
%be given and $f\in H_{L_1,L_2}^p(X_1\times X_2)$. Then we may write $f=g+b$ where $g\in
%H_{L_1,L_2}^{2}( X_1\times X_2 )$ and $b\in H_{L_1,L_2}^{1}( X_1\times X_2 )$ such
% that $\|g\|^{2}_{H_{L_1,L_2}^{2}( X_1\times X_2 )}\le C\alpha^{2-p}\|f\|^p_{H_{L_1,L_2}^p( X_1\times X_2 )}$ and
% $\|b\|_{H_{L_1,L_2}^{1}( X_1\times X_2 )}\le C\alpha^{1-p}\|f\|^p_{H_{L_1,L_2}^p( X_1\times X_2 )}$, where $C$ is an
%absolute constant.
%\end{theorem}

\begin{proof}[Proof of Theorem \ref{theorem C-Z decomposition for Hp}]
By density, we may assume that $f\in H_{L_1,L_2}^p(X_1\times
X_2)\cap H^2(X_1\times X_2) $. Let $\alpha>0$ and set
$\Omega_\ell:=\{(x_1,x_2)\in X_1\times X_2:\,
Sf(x_1,x_2)>\alpha2^\ell \}$, $\ell\geq0$. Set
$$ B_0:=\Big\{ R=I_{\alpha_1}^{k_1}\times I_{\alpha_1}^{k_1}:\, \mu(R\cap \Omega_0)<{1\over 2A_0}\mu(R) \Big\} $$
and
$$ B_\ell:=\Big\{ R=I_{\alpha_1}^{k_1}\times I_{\alpha_1}^{k_1}:\, \mu(R\cap \Omega_{\ell-1})\geq{1\over 2A_0}\mu(R), \mu(R\cap \Omega_\ell)<{1\over 2A_0}\mu(R) \Big\}  $$
for $\ell\geq1$.

By using the reproducing formula and the decomposition (\ref{e2
in section 5.3.3}) as in the proof of
Proposition~\ref{prop-product H-SL subset H-at}, we have
\begin{eqnarray*}
f(x_1,x_2)
&=& \sum_{\ell\in\mathbb{Z}}\sum_{R\in B_\ell}  \int_{T(R)} K_{\psi(t_1\sqrt{L_1})}(x_1,y_1)K_{\psi(t_2\sqrt{L_2})}(x_2,y_2)\nonumber\\
&&\hskip1cm(t_1^2L_1e^{-t_1^2L_1}\otimes t_2^2L_2e^{-t_2^2L_2})(f)(y_1,y_2)\,d\mu_1(y_1)d\mu_2(y_2)\,{dt_1dt_2\over t_1t_2}\nonumber\\
&=& g(x_1,x_2)+b(x_1,x_2),
\end{eqnarray*}
where
\begin{align*}
g(x_1,x_2)
&:= \sum_{R\in B_0}  \int_{T(R)} K_{\psi(t_1\sqrt{L_1})}(x_1,y_1)K_{\psi(t_2\sqrt{L_2})}(x_2,y_2)\nonumber\\
&\hskip1cm(t_1^2L_1e^{-t_1^2L_1}\otimes t_2^2L_2e^{-t_2^2L_2})(f)(y_1,y_2)d\mu_1(y_1)d\mu_2(y_2){dt_1dt_2\over t_1t_2}
\end{align*}
and
\begin{align*}
b(x_1,x_2)
&:= \sum_{\ell>1}\sum_{R\in B_\ell}  \int_{T(R)} K_{\psi(t_1\sqrt{L_1})}(x_1,y_1)K_{\psi(t_2\sqrt{L_2})}(x_2,y_2)\nonumber\\
&\hskip1cm(t_1^2L_1e^{-t_1^2L_1}\otimes t_2^2L_2e^{-t_2^2L_2})(f)(y_1,y_2)d\mu_1(y_1)d\mu_2(y_2){dt_1dt_2\over t_1t_2}.
\end{align*}

As for $g$, by writing $\|g\|_{L^2(X_1\times X_2)}=\sup_{h:\, \|h\|_{L^2}=1 }|\langle g,h\rangle|$, and noting that
\begin{align*}
|\langle g,h\rangle|
&= \Big|\sum_{R\in B_0}  \int_{T(R)} \psi(t_1\sqrt{L_1})\psi(t_2\sqrt{L_2})(h)(y_1,y_2)\nonumber\\
&\hskip1cm (t_1^2L_1e^{-t_1^2L_1}\otimes t_2^2L_2e^{-t_2^2L_2})(f)(y_1,y_2)d\mu_1(y_1)d\mu_2(y_2){dt_1dt_2\over t_1t_2}\Big|\\
&\leq C\|h\|_{L^2} \bigg(\sum_{R\in B_0}  \int_{T(R)}\big|(t_1^2L_1e^{-t_1^2L_1}\otimes t_2^2L_2e^{-t_2^2L_2})(f)(y_1,y_2)\big|^2d\mu_1(y_1)d\mu_2(y_2){dt_1dt_2\over t_1t_2}\bigg)^{1/2},
\end{align*}
we have
\begin{eqnarray*}
\|g\|_{L^2}
\leq C\bigg(\sum_{R\in B_0}  \int_{T(R)}\big|(t_1^2L_1e^{-t_1^2L_1}\otimes t_2^2L_2e^{-t_2^2L_2})(f)(y_1,y_2)\big|^2d\mu_1(y_1)d\mu_2(y_2){dt_1dt_2\over t_1t_2}\bigg)^{1/2}.
\end{eqnarray*}

Also note that
\begin{eqnarray*}
&& \int_{Sf(x_1,x_2)\leq \alpha} Sf(x_1,x_2)^2 d\mu_1(x_1)d\mu_2(x_2)\\
&&= \int_{\Omega_{0}^c} \int_{\Gamma_1(x_1) }\int_{\Gamma_2(x_2) }\big|\big(t_1^2L_1e^{-t_1^2L_1}\otimes
t_2^2L_2e^{-t_2^2L_2}\big)(f)(y_1,y_2)\big|^2\\
&&\hskip2cm{d\mu_1(y_1)d\mu_2(y_2)\ \! dt_1dt_2\over t_1V(x_1,t_1) t_2V(x_2,t_2)} d\mu_1(x_1)d\mu_2(x_2)\\
&&= \int_0^\infty\!\!\int_0^\infty\!\!\int_{X_1\times X_2}\big|\big(t_1^2L_1e^{-t_1^2L_1}\otimes
t_2^2L_2e^{-t_2^2L_2}\big)(f)(y_1,y_2)\big|^2\\
&&\hskip1cm\times\mu(\{(x_1,x_2)\in \Omega_{0}^c:\, d_1(x_1,y_1)<t_1,d_2(x_2,y_2)<t_2\})  \,{d\mu_1(y_1)d\mu_2(y_2)  dt_1dt_2\over t_1V(x_1,t_1) t_2V(x_2,t_2)} \\
&&\geq C\sum_{R\in B_\ell}\int_{T(R)} \big|(t_1^2L_1e^{-t_1^2L_1}\otimes t_2^2L_2e^{-t_2^2L_2})(f)(y_1,y_2)\big|^2\, d\mu_1(y_1)d\mu_2(y_2){dt_1dt_2\over t_1t_2}.
\end{eqnarray*}

As a consequence, we have
\begin{eqnarray*}
\|g\|_{L^2}^2
\leq C\int_{Sf(x_1,x_2)\leq \alpha} Sf(x_1,x_2)^2 d\mu_1(x_1)d\mu_2(x_2).
\end{eqnarray*}

It remains to estimate $\|b\|_{H^1_{L_1,L_2}(X_1\times X_2)}$.
From the definition of the function $b(x_1,x_2)$, we have
\begin{eqnarray*}
&&\|b\|_{H^1_{L_1,L_2}(X_1\times X_2)}\\
&&\leq \sum_{\ell\geq1} \bigg\| \sum_{R\in B_\ell}  \int_{T(R)} K_{\psi(t_1\sqrt{L_1})}(x_1,y_1)K_{\psi(t_2\sqrt{L_2})}(x_2,y_2)\nonumber\\
&&\hskip2cm(t_1^2L_1e^{-t_1^2L_1}\otimes t_2^2L_2e^{-t_2^2L_2})(f)(y_1,y_2)
\,d\mu_1(y_1)d\mu_2(y_2)\,{dt_1dt_2\over t_1t_2} \bigg\|_{H^1_{L_1,L_2}(X_1\times X_2)}.
\end{eqnarray*}

From the proof of Proposition \ref{prop-product H-SL subset
H-at}, we see that, for $\ell\geq1$,
\begin{eqnarray*}
&& {1\over \lambda_\ell}\sum_{R\in B_\ell}  \int_{T(R)} K_{\psi(t_1\sqrt{L_1})}(x_1,y_1)K_{\psi(t_2\sqrt{L_2})}(x_2,y_2)\nonumber\\
&&\hskip2cm(t_1^2L_1e^{-t_1^2L_1}\otimes t_2^2L_2e^{-t_2^2L_2})(f)(y_1,y_2)d\mu_1(y_1)d\mu_2(y_2){dt_1dt_2\over t_1t_2}
\end{eqnarray*}
is an $(H^1_{L_1,L_2},2,N)$-atom, which we denote it by
$a_\ell$, where $\lambda_\ell$ is the coefficient of $a_\ell$
defined by
$$
    \lambda_\ell
    :=C\bigg\|\bigg( \sum_{R\in B_\ell}
        \int_{0}^\infty\!\!\!\int_{0}^\infty \big|(t_1^2L_1e^{-t_1^2L_1}\otimes t_2^2L_2e^{-t_2^2L_2})(f)(y_1,y_2)\big|^2\chi_{T(R)}
        {dt_1dt_2\over t_1t_2}\bigg)^{1/2}\bigg\|_{L^2}\mu(\widetilde{\Omega}_\ell)^{1/2}.
$$
Here we point out that the support of $a_\ell$ is
$\widetilde{\Omega}:=\{(x_1,x_2)\in X_1\times X_2:\,
\mathcal{M}_s(\chi_{\Omega})(x_1,x_2)>1/(2A_0)\}$, where
$\Omega_\ell=\{(x_1,x_2)\in X_1\times X_2:\,
Sf(x_1,x_2)>\alpha2^\ell \}$. Hence, following the same
argument in the proof of Proposition \ref{prop-product H-SL
subset H-at}, we obtain that
$$ |\lambda_\ell|\leq C\alpha2^\ell\mu(\Omega_\ell). $$
Moreover, Lemma \ref{leAtom} implies that $\|a_\ell\|_{H^1_{L_1,L_2}(X_1\times X_2)}\leq C$,
where $C$ is a positive constant independent of $a_\ell$.

As a consequence, we have
\begin{align*}
    \lefteqn{\|b\|_{H^1_{L_1,L_2}(X_1\times X_2)}}\hspace{1cm}\\
    &\leq \sum_{\ell\geq1} |\lambda_\ell|\bigg\| {1\over \lambda_\ell}\sum_{R\in B_\ell}
        \int_{T(R)} K_{\psi(t_1\sqrt{L_1})}(x_1,y_1)K_{\psi(t_2\sqrt{L_2})}(x_2,y_2)\nonumber\\
    &\hskip2cm(t_1^2L_1e^{-t_1^2L_1}\otimes t_2^2L_2e^{-t_2^2L_2})(f)(y_1,y_2)
        d\mu_1(y_1)d\mu_2(y_2){dt_1dt_2\over t_1t_2} \bigg\|_{H^1_{L_1,L_2}(X_1,X_2)}\\
    &\leq C \sum_{\ell\geq1} \alpha2^\ell\mu(\Omega_\ell)\\
    &\leq C \int_{Sf(x_1,x_2)>\alpha} Sf(x_1,x_2) d\mu_1(x_1)d\mu_2(x_2) \\
    &\leq C \alpha^{1-p} \int_{Sf(x_1,x_2)>\alpha} Sf(x_1,x_2)^p d\mu_1(x_1)d\mu_2(x_2) \\
    &\leq C \alpha^{1-p} \|f\|_{H^p_{L_1,L_2}(X_1,X_2)}. \hfill\qedhere
\end{align*}
\end{proof}

We are now ready to prove Theorem \ref{theorem interpolation Hp}.

\begin{proof}[Proof of Theorem \ref{theorem interpolation Hp}]
Suppose that $T$  is bounded from $H_{L_1,L_2}^{1}(X_1\times
X_2)$ to $L^{1}(X_1\times X_2)$ and  from
$H_{L_1,L_2}^{2}(X_1\times X_2)$ to  $L^{2}(X_1\times X_2)$.
For any given $\lambda>0$ and $f\in H_{L_1,L_2}^p(X_1\times
X_2)$, by the Calder\'on--Zygmund decomposition,
$$f(x_1,x_2)=g(x_1,x_2)+b(x_1,x_2)$$ with
$$
    \|g\|^{2}_{H_{L_1,L_2}^{2}(X_1\times X_2)}
    \le C\lambda^{2-p}\|f\|_{H_{L_1,L_2}^p(X_1\times X_2)}^p\,\,\,\ {\rm and\ }\,\,
    \|b\|_{H_{L_1,L_2}^{1}(X_1\times X_2)}
    \le C\lambda^{1-p}\|f\|_{H_{L_1,L_2}^p(X_1\times X_2)}^p.
$$

Moreover, we have already proved the estimates
$$\|g\|^{2}_{H_{L_1,L_2}^{2}(X_1\times X_2)}\le C\int_{Sf(x_1,x_2)\le \alpha}Sf(x_1,x_2)^{2}\,d\mu_1(x_1)d\mu_2(x_2)$$ and
$$\|b\|^{1}_{H_{L_1,L_2}^{1}(X_1\times X_2)}\le C\int_{Sf(x_1,x_2)> \alpha}Sf(x_1,x_2)\,d\mu_1(x_1)d\mu_2(x_2),$$
which imply that
\begin{align*}
    \|Tf\|^p_{L^p(X_1\times X_2)}&=   p\int_0^\infty \alpha^{p-1} \mu(\{(x_1,x_2): |Tf(x_1,x_2)|>\alpha\})d \alpha\\
    &\le  p\int_0^\infty \alpha^{p-1}\mu(\{(x_1,x_2): |Tg(x_1,x_2)|>\alpha/2\})d\alpha\\
    &\hskip.5cm+p\int_0^\infty \alpha^{p-1}\mu(\{(x_1,x_2): |Tb(x_1,x_2)|>\alpha/2\})d\alpha\\
    &\le  p\int_0^\infty \alpha^{p-2-1}\int_{Sf(x_1,x_2)\le \alpha}Sf(x_1,x_2)^{2}\,d\mu_1(x_1)d\mu_2(x_2) d\alpha\\
    &\hskip.5cm+p\int_0^\infty \alpha^{p-1-1}\int_{Sf(x_1,x_2)>\alpha}Sf(x_1,x_2)\,d\mu_1(x_1)d\mu_2(x_2) d\alpha\\
    &\le  C\|f\|^p_{H_{L_1,L_2}^p(X_1\times X_2)}
\end{align*}
for any $1<p<2$. Hence,  $T$ is bounded from $H_{L_1,L_2}^p(X_1\times X_2)$ to $L^p(X_1\times X_2)$.
\end{proof}

%-----------------------------------------------------------------------
\section{The relationship between   $H^p_{L_1,
L_2}({X_1\times X_2})$ and   $L^p({X_1\times X_2})$}
\label{sec:HpandLp}

%First note that under the assumption of Gaussian upper bounds
%(\ref{Gaussian}), following the approaches used in~\cite{HLMMY}
%in the one-parameter setting, we can obtain that
%$H^p_{L_1,L_2}({X_1\times X_2})=L^p({X_1\times X_2})$ for all
%$1<p<\infty$. Second, if one assumes only the Davies--Gaffney
%estimates on the heat semigroups of $L_1$ and $L_2$, then for
%$1<p<\infty$, the product Hardy space $H^p_{L_1,L_2}({X_1\times
%X_2})$ may or may not coincide with the usual Lebesgue space
%$L^p({X_1\times X_2})$. However, it can be verified that
%$H^2_{L_1,L_2}({X_1\times X_2})=H^2({X_1\times X_2})$.

Before proving our main result Theorem \ref{theorem-Hp-Lp}, we
point out that Theorem \ref{theorem-Hp-Lp} is an extension of
Theorem~4.19 in Uhl's PhD thesis~\cite[Section~4.4]{U}. In
Theorem~4.19 (\cite[Section~4.4]{U}), to obtain the coincidence
of the Hardy space and the Lebesgue space, Uhl assumed that $L$
is an injective operator on~$L^2(X)$. Here we note that if $L$
satisfies the {\it generalized Gaussian estimates}~$({\rm
GGE}_{p_0})$ for some $1\leq p_0<2$, then $L$ is injective.
This result seems new and leads to the fact that $H^2(X_1\times
X_2)=L^2(X_1\times X_2)$ (see the proof of
Theorem~\ref{theorem-Hp-Lp} in this section).

\begin{theorem}\label{theorem injective}
    If $L$ satisfies the generalized Gaussian estimates~$({\rm
    GGE}_{p_0})$ for some $p_0$ with $1 \leq p_0 < 2$, then the
    operator $L$ is injective on $L^2(X)$.
\end{theorem}

\begin{proof}
Take $\phi\in L^2(X)$ with $L\phi = 0$. From the functional
calculus,
$$  e^{-tL}-I = \int_0^t {\partial\over \partial s} e^{-sL} ds
= -\int_0^t Le^{-sL} ds. $$
Then we have
$$ (e^{-tL}-I)(\phi)= -\int_0^t Le^{-sL} ds (\phi) =0,$$
which implies that
\begin{eqnarray}\label{identity}
\phi=e^{-tL}\phi
\end{eqnarray}
holds for all $t>0$.  Note that \eqref{identity} is proved in
\cite[page 9]{HLMMY}.

Next, as shown in Lemma~2.6 of~\cite{U}, the {\it generalized
Gaussian estimates}~$({\rm GGE}_{p_0})$ imply the following
$L^2\to L^{p'_0}$ off-diagonal estimates:
\begin{equation}\label{eqn:2p0estimate}
    \|P_{B(x,\sqrt{t})}e^{-tL}P_{C_j(x,\sqrt{t})}\|_{2\to p_0'}
    \leq CV(x,\sqrt{t})^{-(1/2-1/p_0')}e^{-c4^j},
\end{equation}
where $C_j(x,r) := B(x,2^jr)\setminus B(x,2^{j-1}r)$ for $j \geq 1$
and $C_0(x,r) = B(x,r)$.

As a consequence of Fatou's lemma, \eqref{identity} and
\eqref{eqn:2p0estimate}, we have that
\begin{align*}
    \|\phi\|_{p_0'}&\leq\lim_{t\to \infty}\|P_{B(x,\sqrt{t})}\phi\|_{p_0'}
    =\lim_{t\to \infty}\|P_{B(x,\sqrt{t})}e^{-tL}\phi\|_{p_0'}\\
    &\leq \lim_{t\to \infty} \sum_{j=0}^\infty\|P_{B(x,\sqrt{t})}e^{-tL}P_{C_j(x,\sqrt{t})} \phi\|_{p_0'}\\
    &\leq \lim_{t\to \infty} \sum_{j=0}^\infty C V(x,\sqrt{t})^{-(1/2-1/p_0')}e^{-c4^j}\|\phi\|_{2}\\
    &\leq \lim_{t\to \infty}CV(x,\sqrt{t})^{1/p_0'-1/2}\|\phi\|_{2}\\
    &=0.
\end{align*}
Here in the final step we have used the fact that
$\mu(X)=\infty$. Thus, we obtain that $\phi = 0$ a.e. This
completes the proof of Theorem~\ref{theorem injective}.
\end{proof}

Next, we give a vector-valued version of a theorem about the
area function associated with an operator $L$ in the
one-parameter setting.

Suppose $L$ is a non-negative self-adjoint operator defined on
$L^2(X;H)$, where $H$ is a Hilbert space with a norm
$|\cdot|_H$. Moreover, assume that $L$ satisfies the {\it
generalized Gaussian estimates}  $({\rm GGE}_{p_0})$
%\begin{align}\label{generalGE_chen2}
%\|P_{B(x,t^{1/2})}e^{-tL}P_{B(y,t^{1/2})}\|_{L^{p_0}(X;H)\to L^{p_0'}(X;H)}
%\leq C V(x,t^{1/2})^{-(1/p_0-1/p_0')}\exp\Big(-b\frac{d(x,y)^2}{t}\Big)
%\end{align}
for some $p_0$ with $1\leq p_0< 2$.

We now define an area function $S_H:
L^2(X;H)\rightarrow L^2(X)$ associated with $L$ by
\begin{eqnarray*}%\label{esf_chen1}
    \hskip.7cm S_Hf(x)
    := \bigg(\int_{\Gamma(x) }\big|\big(
        t^2Le^{-t^2L}\big)f(y)\big|_H^2\
        {d\mu(y) \ \! dt\over tV(x,t)}\bigg)^{1/2}.
   \end{eqnarray*}
Then we prove the following boundedness result for $S_H$.

\begin{theorem}\label{theorem vector area function}
    Suppose that $L$ is a non-negative self-adjoint
    operator defined on $L^2(X;H)$ satisfying the generalized
    Gaussian estimates  $({\rm GGE}_{p_0})$ for some
    $p_0\in[1,2)$. Then there exists a positive constant $C$ such
    that
    \begin{eqnarray}\label{boundedness of vector area function}
    \|S_Hf\|_{L^p(X)}\leq C\|f\|_{L^p(X;H)}
    \end{eqnarray}
    for all $p\in (p_0,p_0')$ and all $f\in L^p(X;H)\cap L^2(X;H)$.
\end{theorem}

\begin{proof}
This boundedness result (\ref{boundedness of vector area
function}) is a vector-valued version of the result (4.15) in
Uhl's PhD thesis \cite[Section 4.4]{U}. We restate Uhl's proof
in our vector-valued setting.

\textbf{Step I.} We first prove that $\|S_Hf\|_{L^p(X)}\leq
C\|f\|_{L^p(X;H)}$ for $p_0<p\leq 2$.

To see this, we define
$$
g^*_{\lambda,H} f(x):=
\left(\int_0^\infty \int_X \Big(\frac{t}{d(x,y)+t}\Big)^{n\lambda}\big|\big(
        t^2Le^{-t^2L}\big)f(y)\big|_H^2\,\frac{d\mu(y)dt}{tV(x,t)}\right)^{1/2},
$$
where $n$ is the upper dimension of the doubling measure $\mu$.
Then it is easy to see that $\|S_Hf\|_{L^p(X)}\leq
C\|g^*_{\lambda,H}f\|_{L^p(X)}$ for each $\lambda>1$. Thus, it
suffices to prove that for each for $p$ with $p_0 < p \leq 2$,
there exists a positive constant $C$ such that
$\|g^*_{\lambda,H}f\|_{L^p(X)}\leq C\|f\|_{L^p(X;H)}$ for all
$f\in L^p(X;H)$. We do so by interpolation.

We first show the $L^2$ boundedness of $g^*_{\lambda,H}f$. To
see this, we point out that by Fubini's Theorem,
$$
\int_F |g^*_{\lambda,H}f|^2 d\mu(x)=
\int_0^\infty\int_X J_{\lambda,F}(y,t)\big|\big(
        t^2Le^{-t^2L}\big)f(y)\big|_H^2\,\frac{d\mu(y)dt}{t},
$$
with
$$
J_{\lambda,F}(y,t)=\int_F \Big(\frac{t}{d(x,y)+t}\Big)^{D\lambda} \frac{d\mu(x)}{V(x,t)},
$$
which holds for any closed set $F\subset X$.

Then we have the estimate
$$J_{\lambda,F}(y,s)\leq C_\lambda,$$
where $C_\lambda$ is a constant depending only on $\lambda$ and
$n$ but not on $F$, $y$ or $s$. This estimate follows directly from the
inequality~(4.16) in Uhl's PhD thesis~\cite[Section 4.4]{U}.

As a consequence, we obtain that
$$
    \|g^*_{\lambda,H}f\|_2^2
    \leq C_\lambda\! \int_0^\infty\!\!\!\int_X \big|\big(
        t^2Le^{-t^2L}\big)f(y)\big|_H^2\frac{d\mu(y)dt}{t}
    \leq C_\lambda\! \int_0^\infty t^4e^{-2t^2}\frac{dt}{t}\|f\|_{L^2(X;H)}^2
    \leq C\|f\|_{L^2(X;H)}^2.
$$

Next we point out that $g^*_{\lambda, H}$ is weak-type
$(p_0,p_0)$. All the calculations and ingredients of Uhl's
proof in~\cite[pp.63--74]{U} of this fact for $L^{p_0}(X)$,
namely the use of the Calder\'on--Zygmund decomposition, the
$L^2$-integral, duality in the sense of $L^2$, the
Hardy--Littlewood maximal operator, and the $L^{p_0}\to L^2$
estimate, go through in our vector-valued
setting~$L^{p_0}(X;H)$. Thus we need only apply the rest of
Uhl's proof, replacing the absolute value $|\cdot|$ used there
by our norm~$|\cdot|_H$.

\textbf{Step II.} We now prove that $\|S_Hf\|_{L^p(X)}\leq
C\|f\|_{L^p(X;H)}$ for $2\leq p< p_0'$.

To see this, we consider the Littlewood--Paley $g$-function
defined by
$$
    G_H f(x)
    := \left(\int_0^\infty|t^2Le^{-t^2L}f(x)|_H^2\frac{dt}{t}\right)^{1/2}.
$$
We claim that
\begin{equation}\label{eqn:GH}
    \|G_Hf\|_{L^p(X)}
    \leq C\|f\|_{L^p(H)}.
\end{equation}
The proof of~\eqref{eqn:GH} is exactly the
same as that of the proof for the Euclidean, non-vector-valued
case in Auscher's paper~\cite[Section~7.1]{A}. The key
ingredient of Auscher's proof is Theorem~2.2 of~\cite{A}. It is
noted in~\cite[Remark~7, after Theorem~2.2]{A} that Theorem~2.2
also holds in the vector-valued case. Further, the proof of
Theorem~2.2 in Auscher's paper goes through in the case of
spaces of homogeneous type.

Auscher's proof of~\eqref{eqn:GH} requires the Davies--Gaffney
estimates and \eqref{eqn:2p0estimate}.
%
%the following $L^2\to L^{p'_0}$ off-diagonal
%estimates:
%\begin{equation}\label{eqn:2p0estimate}
%    \|P_{B(x,\sqrt{t})}e^{-tL}P_{C_j(x,\sqrt{t})}\|_{2\to p_0'}
%    \leq CV(x,\sqrt{t})^{-(1/2-1/p_0')}e^{-c4^j},
%\end{equation}
%where $C_j(x,r) := B(x,2^jr)\setminus B(x,2^{j-1}r)$ for $j \geq 1$
%and $C_0(x,r) = B(x,r)$.
%
The Davies--Gaffney estimates are one of
our hypotheses. The estimate~\eqref{eqn:2p0estimate} follows from
the generalized Gaussian estimates~$({\rm GGE}_{p_0})$, as is
shown in Lemma~2.6 of~\cite{U}. Thus inequality~\eqref{eqn:GH}
holds.

Then, following the duality argument in Uhl's
proof~\cite[pp.74--75]{U}, we obtain that for all $\phi\in
L^{(p/2)'}(X)$,
$$
    |\langle (S_Hf)^2,\phi  \rangle|
    \leq |\langle (G_Hf)^2, \mathcal{M}(|\phi|) \rangle|.
$$
Therefore $\|S_Hf\|_{L^p(X)} \leq \|G_Hf\|_{L^p(X)} \leq
C\|f\|_{L^p(H)}$, as required.
\end{proof}

\begin{remark}
    We point out that in Step~II in the proof above, we can obtain
    the following result as well:
    $\|S_{H,\psi}f\|_{L^p(X)}\leq\|G_{H,\psi}f\|_{L^p(X)}\leq
    C\|f\|_{L^p(X;H)}$, where $\psi$ appears in the reproducing
    formula in \eqref{e2 in section 5.3.3}, and
    \begin{eqnarray*}
        \hskip.7cm S_{H,\psi}f(x)
        := \bigg(\int_{\Gamma(x) }\big|\big(
            \psi(t\sqrt{L})\big)f(y)\big|_H^2\
            {dy \ \! dt\over tV(x,t)}\bigg)^{1/2},
            \end{eqnarray*}
    and
    $$
    G_{H,\psi}f(x)=\left(\int_0^\infty|\psi(t\sqrt{L})f(x)|_H^2\frac{dt}{t}\right)^{1/2}.
    $$
\end{remark}

Now we can prove Theorem~\ref{theorem-Hp-Lp}.

\begin{proof}[Proof of Theorem \ref{theorem-Hp-Lp}]
Note that Part (ii) is a consequence of Part (i) and
Theorem~\ref{theorem interpolation Hp}. It suffices to prove Part (i).

By Theorem \ref{theorem
injective} we obtain that $L_1$ and $L_2$ are injective
operators on $L^2(X_1)$ and $L^2(X_2)$, respectively. As a
consequence, the null space $N(L_1\otimes L_2)=\{0\}$, which
yields that $H^2(X_1\times X_2)=L^2(X_1\times X_2)$ since
$L^2(X_1\times X_2)=H^2(X_1\times X_2) \oplus N(L_1\otimes
L_2)$.
%Hence $H^2(X_1\times X_2)\cap L^p(X_1\times X_2)$ is
%dense in $L^p(X_1\times X_2)$ and $H^2(X_1\times X_2)$.
%
Thus,
to prove $H^p_{L_1,L_2}(X_1\times X_2)=L^p(X_1\times X_2)$ for $p_0<p\leq 2$, it suffices to prove
that for all $f\in L^2(X_1\times X_2)\cap L^p(X_1\times X_2)$,
\begin{align}\label{part i}
\|f\|_{L^p(X_1\times X_2)}\leq C\|Sf\|_{L^p(X_1\times X_2)} \leq
C\|f\|_{L^p(X_1\times X_2)}.
\end{align}
And then the result $H^p_{L_1,L_2}(X_1\times X_2)=L^p(X_1\times X_2)$ for $2<p< p'_0$
follows from the duality argument. This implies that Part (i) holds.

We now verify \eqref{part i}. First, write the area function as
\begin{eqnarray*}
&&\Big(\int_{\Gamma(x)}
|t_1^2L_1e^{-t_1^2L_1}\otimes t_2^2L_2e^{-t_2^2L_2} f(y)|^2
{d\mu_1(y_1)\over V(x_1,t_1)}{dt_1\over t_1}{d\mu_2(y_2)\over
V(x_2,t_2)}{dt_2\over
t_2}\Big)^{1/2}\\
&&=\left(\int_{\Gamma(x_1)}\left[\int_{\Gamma(x_2)}
|\big(t_1^2L_1e^{-t_1^2L_1}F_{t_2,y_2}\big)(y_1)|^2
{d\mu_2(y_2)\over V(x_2,t_2)}{dt_2\over t_2}\right]{d\mu_1(y_1)\over
V(x_1,t_1)}{dt_1\over t_1}\right)^{1/2}
\end{eqnarray*}
where $F_{t_2,y_2}(\cdot)=\big(t_2^2L_2e^{-t_2^2L_2}
f\big)(\cdot,y_2)$.

For each $x_2\in X_2$, we define the Hilbert-valued function
space $L^2(X_1;H_{x_2})$ via the following $H_{x_2}$ norm:
$$
|G_{t_2,y_2}(y_1)|_{H_{x_2}}:=\left[\int_{\Gamma(x_2)}
|G_{t_2,y_2}(y_1)|^2 {d\mu_2(y_2)\over V(x_2,t_2)}{dt_2\over
t_2}\right]^{1/2}.
$$
%Then function $F_{t_2,y_2}(\cdot)=\big((t_2^2L_2)^{K}e^{-t_2^2L_2}
%f\big)(\cdot,y_2)$ defines the $H_{x_2}$-valued function on
%$L^2(X_1;H_{x_2})$.
Then $L_1$ can be extended to act on $L^2(X_1;H_{x_2})$ in a
natural way. Also the generalized Gaussian estimates can
be extended to the semigroup $e^{tL}$ acting on
$L^2(X_1;H_{x_2})$. That is, by Minkowski's inequality
\begin{eqnarray*}
\lefteqn{\|P_{B(x_1,t^{1/2})}e^{-tL_1}P_{B(y_1,t^{1/2})}G_{t_2,y_2}(\cdot)\|_{L^{p_0'}(X;H)}}\\
&&=\Big\||P_{B(x_1,t^{1/2})}e^{-tL_1}P_{B(y_1,t^{1/2})}G_{t_2,y_2}(\cdot)|_{H}\Big\|_{L^{p_0'}(X_1)}\\
&&\leq
\Big|\|P_{B(x_1,t^{1/2})}e^{-tL_1}P_{B(y_1,t^{1/2})}G_{t_2,y_2}(\cdot)\|_{L^{p_0'}(X_1)}\Big|_{H}\\
&&\leq C
V(x_1,t^{1/2})^{-(1/p_0-1/p_0')}\exp\Big(-b\frac{d(x_1,y_1)^2}{t}\Big)\big|\|G_{t_2,y_2}\|_{L^p(X_1)}\big|_{H}\\
&&\leq C
V(x_1,t^{1/2})^{-(1/p_0-1/p_0')}\exp\Big(-b\frac{d(x_1,y_1)^2}{t}\Big)\big\||G_{t_2,y_2}|_H\big\|_{L^p(X_1)}\\
&&=C
V(x_1,t^{1/2})^{-(1/p_0-1/p_0')}\exp\Big(-b\frac{d(x_1,y_1)^2}{t}\Big)\|G_{t_2,y_2}\|_{L^p(X;H)}.
\end{eqnarray*}
Define the area function $S_{H_{x_2}}$ from $L^2(H_{x_2})$ to
$L^2(X_1)$ by
\begin{eqnarray*}%\label{esfchen2}
    \hskip.7cm S_{H_{x_2}}G_{t_2,y_2}(x_1)
    := \bigg(\int_{\Gamma(x_1) }\big|\big(
        t^2L_1e^{-t^2L_1}\big)G_{t_2,y_2}(y_1)\big|_H^2\
        {d\mu_1(y_1) \ \! dt\over tV(x_1,t)}\bigg)^{1/2}.
   \end{eqnarray*}
Recall that $F_{t_2,y_2}(\cdot)=\big(t_2^2L_2e^{-t_2^2L_2}
f\big)(\cdot,y_2)$. So by Theorem \ref{theorem vector area
function}, we have for all $p\in (p_0,p_0')$ that
\begin{eqnarray*}
\|Sf\|_{L^p(X_1\times
X_2)}&=&\|S_{H_{x_2}}F_{t_2,y_2}(x_1)\|_{L^p(X_1\times X_2)}\\
&\leq&\|\|F_{t_2,y_2}\|_{L^p(H_{x_2})}\|_{L^p(X_2)}\\
&=&\|\||F_{t_2,y_2}|_{H_{x_2}}\|_{L^p(X_2)}\|_{L^p(X_1)}\\
&=&\Bigg\|\bigg\|\left[\int_{\Gamma(x_2)} |F_{t_2,y_2}(y_1)|^2
{d\mu_2(y_2)\over V(x_2,t_2)}{dt_2\over
t_2}\right]^{1/2}\bigg\|_{L^p(X_2)}\Bigg\|_{L^p(X_1)}\\
&=&\Bigg\|\bigg\|\left[\int_{\Gamma(x_2)}
|\big((t_2^2L_2)e^{-t_2^2L_2} f\big)(y_1,y_2)|^2 {d\mu_2(y_2)\over
V(x_2,t_2)}{dt_2\over
t_2}\right]^{1/2}\bigg\|_{L^p(X_2)}\Bigg\|_{L^p(X_1)}\\
&\leq& C\|f\|_{L^p(X_1\times X_2)}.
\end{eqnarray*}

We can obtain the other direction, that is,
$\|f\|_{L^p(X_1\times X_2)} \leq C\|Sf\|_{L^p(X_1\times X_2)}$,
by using the reproducing formula and then the standard duality
argument and the $L^p$-boundedness of the area function for $2
\leq p < p_0'$. This completes the proof of
Theorem~\ref{theorem interpolation Hp}.
\end{proof}

\bigskip

{\bf Acknowledgments.} P.~Chen, X.T.~Duong, J.~Li and L.A.~Ward
are supported by the Australian Research Council (ARC) under
Grant No.~ARC-DP120100399.  L.X.~Yan is supported by the NNSF
of China, Grant Nos.~10925106 and~11371378. Part of this work
was done during L.X.~Yan's stay at Macquarie University and
visit to the University of South Australia. L.X.~Yan would like
to thank Macquarie University and the University of South
Australia for their hospitality.

\end{document}